\DeclareMathAlphabet{\mathpzc}{OT1}{pzc}{m}{it}

\documentclass[headsepline=true]{scrartcl}
\usepackage{amsmath,amsthm, amssymb, amsfonts, amscd, bbm}
\usepackage{mathtools}
\usepackage[top=1.2in,bottom=1.2in,left=1in,right=1in]{geometry}
\usepackage[table]{xcolor}

\definecolor{darkblue}{rgb}{0,0,0.7}

\usepackage[pdftex,hyperfootnotes=false,colorlinks=true,linkcolor=darkblue,
citecolor=purple,filecolor=magenta,urlcolor=darkgray,pdfborder={0 0 0}]{hyperref}
\usepackage{array}

\usepackage{abstract}
\usepackage{amscd}
\usepackage{tensor}
\usepackage{mathrsfs}
\usepackage{arydshln}
\usepackage{tikz}
\usetikzlibrary{decorations.markings}
\usetikzlibrary{decorations.pathreplacing}
\usepackage{lscape}
\usepackage{enumerate}
\usepackage[capitalize]{cleveref}
\usepackage[utf8]{inputenc}
\usepackage{autonum}
\usepackage{footnote}
\usepackage{wrapfig}

\allowdisplaybreaks[2]

\title{A symmetric Bloch--Okounkov theorem}
\subtitle{}
\author{Jan-Willem M. van Ittersum%
\thanks{\emph{Email}: \href{mailto:j.w.m.vanittersum@uu.nl}{j.w.m.vanittersum@uu.nl}, \newline
Mathematisch Instituut, Universiteit Utrecht, Postbus 80.010, 3508 TA Utrecht, The Netherlands, \newline
Max-Planck-Institut f\"ur Mathematik, Vivatsgasse 7, 53111 Bonn, Germany.
}}

\newcommand{\sltwo}{\mathfrak{sl}_2}
\newcommand{\sltwoz}{\mathrm{SL}_2(\z)}
\newcommand{\one}{\mathbf{1}_n}

\DeclareMathOperator{\height}{ht}
\DeclareMathOperator{\BST}{BST}

\renewcommand{\vec}{\underline}

\newcommand{\n}{\mathbb{N}}
\newcommand{\z}{\mathbb{Z}}
\newcommand{\q}{\mathbb{Q}}
\renewcommand{\r}{\mathbb{R}}
\newcommand{\SA}{\mathcal{S}}
\newcommand{\TA}{\mathcal{T}}
\newcommand{\pdv}[2]{\frac{\partial #1}{\partial #2}}

\newcommand{\cp}{\,|\,}

\theoremstyle{plain} 
\newtheorem{thm}{Theorem}[section]
\newtheorem{lem}[thm]{Lemma} 
\newtheorem{cor}[thm]{Corollary} 
\newtheorem{prop}[thm]{Proposition}

\theoremstyle{definition} 
\newtheorem{defn}[thm]{Definition}
\newtheorem{exmp}[thm]{Example}

\theoremstyle{remark}
\newtheorem{remarkx}[thm]{Remark}
\newenvironment{remark}
  {\pushQED{\qed}\remarkx}
  {\popQED\endremarkx}

\newcommand{\coef}{\mathpzc{B}}
\newcommand{\Eis}{G}

\newcommand{\QtoP}{\q^{\partitions}}
\newcommand{\modular}{M}
\newcommand{\quasimodular}{\widetilde{M}}
\newcommand{\partitions}{\mathscr{P}}

\newcommand{\Faulhaber}{\mathpzc{F}}

\newcommand{\Sym}{\mathrm{Sym}}
\DeclareMathOperator*{\conv}{\ast}
\newcommand{\bigconv}{\mathop{\hspace{-0.0pt}\scalebox{2.4}{\raisebox{-0.3ex}{$\ast$}}}}
\newcommand{\blok}{\odot}
\newcommand{\bigblok}{\bigodot}
\newcommand{\mob}{\mu(\alpha,\mathbf{1})}

\renewcommand{\=}{\: =\: }
\newcommand{\defis}{\: :=\: }
\newcommand{\+}{\,+\,}
\newcommand{\meno}{\,-\,}

\renewcommand{\phi}{\varphi}

\DeclareRobustCommand{\stirling}{\genfrac\{\}{0pt}{}}
\newcommand{\moller}{\mathcal{M}}

\begin{document}
\maketitle
\begin{abstract}%
The algebra of so-called shifted symmetric functions on partitions has the property that for all elements a certain generating series, called the~$q$-bracket, is a quasimodular form. More generally, if a graded algebra~$A$ of functions on partitions has the property that the~$q$-bracket of every element is a quasimodular form of the same weight, we call~$A$ a quasimodular algebra. We introduce a new quasimodular algebra~$\TA$ consisting of symmetric polynomials in the part sizes and multiplicities. %Throughout we apply the main principle that all notions initially defined on quasimodular forms, such as their derivative and Rankin--Cohen brackets, are defined on elements of~$\TA$.
\end{abstract}
%\tableofcontents

\section{Introduction}
Partitions of integers are related in interesting ways to modular forms, starting with the observation that the generating series of partitions is closely related to the Dedekind~$\eta$-function, i.e.
\[\sum_{\lambda \in \partitions} q^{|\lambda|}\=\prod_{n>0} (1-q^n)^{-1}\=q^{1/24}\eta(\tau)^{-1} \quad \quad (q=e^{2\pi i\tau}),\]
where~$\partitions$ denotes set of all partitions and~$|\lambda|$ denotes the integer~$\lambda$ is a partition of. Another example is the occurrence of modular forms in the proof of the partition congruences which go back to Ramanujan \cite{AO01}. 

More recently, partitions were connected to (quasi)modular forms via the~$q$-bracket. Given a function~$f:\partitions\to\q$, the~\emph{$q$-bracket} of~$f$ is defined as the following power series
\begin{equation}\label{eq:qbrac}
\langle f \rangle_q \= \frac{\sum_{\lambda\in\partitions}f(\lambda) q^{|\lambda|}}{\sum_{\lambda\in\partitions} q^{|\lambda|}} \in \q[[q]].
\end{equation}
Before continuing, note that it is not surprising at all that for a well-chosen function~$f$ the~$q$-bracket~$\langle f \rangle_q$ is a quasimodular form, since it is easily seen that the map~\eqref{eq:qbrac} from~$\q^\partitions$ to~$\q[[q]]$ is surjective. What \emph{is} surprising is that one can find graded subalgebras~$A$ of~$\q^\partitions$ which~(i) are `interesting' in the sense that they have an interpretation in combinatorics, enumerative geometry or another field of mathematics and~(ii) have the property that the~$q$-bracket of a homogeneous function~$f\in A$ is quasimodular of the same weight as~$f$. %(possibly of some level~$N$)
In this case we call~$A$ a \emph{quasimodular algebra}\label{defn:quasimodularalgebra}. %(of level~$N$). 
Note that the~$q$-bracket is linear but not multiplicative, so in order to show that an algebra is quasimodular, it is not sufficient to show that the~$q$-brackets of the generators of such an algebra are quasimodular. The aim of this paper is to introduce new quasimodular algebras. %and to give applications.

The Bloch--Okounkov theorem \cite[Theorem 0.5]{BO00} provided the first quasimodular algebra~$\Lambda^*$. Write a partition~$\lambda$ as a non-increasing sequence~$(\lambda_1,\lambda_2,\ldots)$ of non-negative integers with~$|\lambda|=\sum_{i=1}^\infty \lambda_i$ finite. The~$\q$-algebra~$\Lambda^*$ is freely generated by the so-called \emph{shifted symmetric power sums}
\begin{equation}\label{eq:qk} Q_k(\lambda) \= c_k \+ \sum_{i=1}^\infty \bigl((\lambda_i-i+\tfrac{1}{2})^{k-1}-(-i+\tfrac{1}{2})^{k-1}\bigr) \quad \quad (k\geq 2),\end{equation}
where the~$c_k$ are constants given by~$\frac{1}{x}+\sum_{k} c_k \frac{x^{k-1}}{(k-1)!} = \frac{1}{2\sinh(x/2)}$
%(they arise from renormalizing the diverging series~$\sum_{i=1}^\infty(-i+1/2)^{k-1}$)
. The function~$Q_3$ naturally occurs  in the simplest case of the Gromov--Witten theory of an elliptic curve, as discovered by Dijk\-graaf \cite{Dij95} and for which quasimodularity was proven rigorously in \cite{KZ95}. Quasimodularity of~$\Lambda^*$ is used in many recent works in enumerative geometry \cite{CMZ16, BK18, CMSZ19, MIL19,GM16}. There are many other functions in invariants of partitions which turn out to be elements of~$\Lambda^*$, for example symmetric polynomials in de modified Frobenius coordinates \cite[Eqn~19]{Zag16}; the hook-length moments \cite[Thm.\ 13.5]{CMZ16} (see \S\S\ref{sec:hlm}); central characters of the symmetric group \cite[Prop.~3]{KO94} and symmetric polynomials in the content vector of a partition \cite[Proof of Thm. 4]{KO94}. 

Previously, the Bloch--Okounkov algebra~$\Lambda^*$ and some generalisations to higher levels (see e.g. \cite{EO06, Eng17}), were the only known quasimodular algebras. However, there are many examples of functions on partitions admitting a quasimodular~$q$-bracket (and in general not belonging to~$\Lambda^*$) \cite[section 9]{Zag16}, for example the M\"oller transformation of functions with quasimodular~$q$-bracket (defined by~\cite[Eqn~45]{Zag16} and recalled in Section~\ref{sec:application}), invariants~$\mathcal{A}_P$ for every even polynomial defined in terms of the arm- and leg-lengths of a partition and the moment functions \begin{align}\label{eq:defSk}
\hspace{75pt} S_k(\lambda) \= -\frac{B_k}{2k}\+\sum_{i=1}^\infty \lambda_i^{k-1} \quad\quad\quad (k \text{ even}, B_k = k\text{th Bernoulli number})\end{align}
that also occur in the study of so-called spin Hurwitz numbers in the algebra of \emph{supersymmetric polynomials}  \cite{EOP08} (in that reference, these functions are only evaluated at strict partitions --- partitions without repeated parts --- and quasimodularity is shown for a correspondingly adapted~$q$-bracket).%, but this will not be our main concern in the current paper.

In this paper, we prove the stronger result that the algebra~$\SA$ generated by these moment functions~$S_k$ is quasimodular. Moreover, besides the pointwise product of functions on partitions, we define a second associative product~$\blok$, called the \emph{induced product} as it is inherited from the product of power series. The vector space~$\mathrm{Sym}^\blok(\SA)$ generated by the elements in~$\SA$ under the induced product is strictly bigger than~$\SA$, is a quasimodular algebra for either of the two products, and has a particularly nice description in terms of functions~$T_{k,l}$ depending not only on the parts of a partition, but also on their multiplicities. Here, the \emph{multiplicity}~$r_m(\lambda)$ of parts of size~$m$ in a partition~$\lambda$ is defined as the number of parts of~$\lambda$ of size~$m$. More precisely, let~$\Faulhaber_l$ be the \emph{Faulhaber polynomial} of positive integer degree~$l$, defined by $\Faulhaber_l(n)=\sum_{i=1}^n i^{l-1}$ for all $n\in\z_{>0}$. Then,~$T_{k,l}$ is given by
\begin{equation}\label{def:Tkl0}\hspace{50pt} T_{k,l}(\lambda) \=  C_{k,l}\+\sum_{m=1}^\infty m^k \Faulhaber_l(r_m(\lambda)) \quad\quad\quad (k\geq 0, l\geq 1, k+l \text{ even})\end{equation}
with~$C_{k,l}$ a constant equal to~$-\frac{B_{k+l}}{2(k+l)}$ if~${k=0}$ or~${l=1}$ and~$0$ else. Let~$\TA$ be the algebra generated by all these~$T_{k,l}$ under the pointwise product. 

We show that~$\mathrm{Sym}^\blok(\SA)$ and~$\TA$ are algebras for the pointwise product as well as for the induced product. In fact, the expression of elements of~$\mathrm{Sym}^\blok(\SA)$ in terms of the~$T_{k,l}$ implies that~$\mathrm{Sym}^\blok(\SA)$ is a strict subalgebra of~$\TA$ (with respect to both products). Our main result is the following:
\begin{thm}\label{thm:1}
The algebras~$\mathrm{Sym}^\blok(\SA)$ and~$\TA$ are quasimodular algebras with respect to the induced product.
\end{thm}
With respect to the pointwise product, these algebras are not quasimodular because of the following subtlety: the $q$-bracket of a homogeneous function~$f$ in~$\TA$ (with respect to the pointwise product) often is of mixed weight (i.e., a linear combination of quasimodular forms of weights bounded by the weight of~$f$). By making use of the induced product, one can explain these lower weight quasimodular forms, as we do in \cref{sec:2}. For example,
\[ \langle T_{0,2}^2 \rangle_q \= G_2^2 \+ \frac{5}{6}G_4 \+ \frac{1}{6}G_2 \+ \frac{1}{288},\]
where $G_2$ and $G_4$ are the Eisenstein series defined by~\eqref{def:eis}. The right-hand side is a quasimodular form of mixed weight, which is explained by the fact that
\[ T_{0,2}^2 \= T_{0,2}\blok T_{0,2} \+ \frac{5}{6}T_{0,4} \+ \frac{1}{6}T_{0,2} \+ \frac{1}{288},\]
is a linear combination of elements of~$\TA$ of different weights with respect to the induced product. 

A main theme throughout this paper is the principle to establish all identities in~$\q^\partitions$ or~$\TA$ \emph{before} taking the~$q$-bracket, instead of doing these computations in~$\q[[q]]$ or the space of quasimodular forms~$\widetilde{M}$. By doing so, we discover the algebraic structure of~$\TA$. Without having the induced product at one's disposal, for example when studying the shifted symmetric algebra~$\Lambda^*$, this seems impossible. See the following table for an overview of situations where the principle is applied:
\vspace{7pt}
\renewcommand{\arraystretch}{1.5}
\begin{savenotes}
\begin{center}\vspace{-10pt}
\begin{tabular}{l l l}
\textbf{previous definitions and results} &\textbf{definitions and results in this work} &\S\S \\ \hline
\emph{multiplication} in~$\q[[q]]$ &\emph{induced product}~$\blok$ on~$\q^\partitions$ &\ref{par:3.2} \\\hline
\emph{$q$-bracket}:~$\q^\partitions\to\q[[q]]$ &\emph{$\vec{u}$-bracket}:~$\q^\partitions \to \q[[u_1,u_2,\ldots]]$ &\ref{par:3.2} \\\hline
\emph{connected~$q$-bracket}:~$\Sym^\otimes(\q^\partitions)\to \q[[q]]$ &\emph{connected product}:~$\Sym^\otimes(\q^\partitions)\to \q^\partitions$  &\ref{par:3.2}\\\hline
\emph{derivative}~$q\frac{\mathrm{d}}{\mathrm{d}q}$ on~$\q[[q]]$ &\emph{derivative} on~$\q^\partitions$ &\ref{par:der} \\\hline
\emph{$\sltwo$-action} on~$\widetilde{M}$ &\emph{$\sltwo$-action} on~$\TA$ &\ref{par:eq}\\\hline
\emph{Rankin--Cohen brackets} on~$\widetilde{M}$ &\emph{Rankin--Cohen brackets} on~$\TA$ &\ref{par:rc} \\\hline
formula for~$\langle H_p f\rangle_q$ in \cite[Eqn~152]{CMZ16}\footnote{In that work the hook-length moment~$H_p$ (see also \S\S\ref{sec:hlm}) was denoted by~$T_{p-1}$.} &formula for~$ T_{k,l}f$ &\ref{par:Tklf} 
\\\hline		
\end{tabular} 
\end{center}
\end{savenotes}
\vspace{5pt}
\renewcommand{\arraystretch}{1}

A further main result of the paper is the following:
\begin{thm}\label{thm:2}
The~$q$-bracket is an equivariant mapping~$\TA\to \widetilde{M}$ with respect to~$\sltwo$-actions by derivations on both spaces. 
\end{thm}

Motivated by the fact that many functions in invariants of partitions are elements of~$\Lambda^*$, in~\cref{sec:application} we describe many functions on partitions which are elements of~$\TA$ or are closely related. Among those are the border strip moments, generalizing the hook-length moments, which are defined in terms of the representation theory of the symmetric group. The corresponding space~$\mathcal{X}$ of border strip moments is the image of a space~$\mathcal{U}$ under the aforementioned M\"oller transform~$\mathcal{M}$,
\begin{wrapfigure}{r}{0.27\textwidth}
  \vspace{-16pt}
 \begin{flushright}
\includegraphics[scale=1]{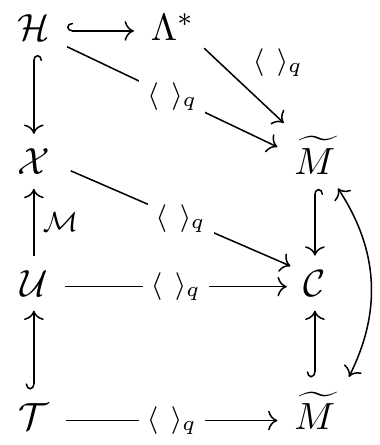}
 \end{flushright}
   \vspace{-16pt}
\end{wrapfigure}
 where $\mathcal{U}$ is generated by the double moment functions $T_{k,l}\in \mathcal{T}$ as well as the odd double moments functions (those for which $k+l$ is odd). The~$q$-brackets of these functions are contained in the space~$\mathcal{C}$ of so-called \emph{combinatorial Eisenstein series}, having the space of quasimodular forms as a subspace. Moreover, the space of hook-length moments~$\mathcal{H}$ is contained in both~$\Lambda^*$ and~$\mathcal{X}$---this contrasts the situation for~$\mathcal{T}$, which by \cref{rk:comparison} has a trivial intersection with~$\Lambda^*$. See the commutative diagram on the right for an overview of the spaces related to~$\TA$ with their corresponding mappings.

We hope that this work---besides advocating the notion of a `quasimodular algebra' by giving a new example of such an algebra and studying its algebraic structure---may serve as a tool for enumerative geometers trying to show that generating series are quasimodular forms. 

The contents of the paper are as follows. In \cref{sec:pre} we recall notions (known to the experts) related to quasimodular forms, partitions and special families of polynomials. Next, in \cref{sec:defn} we motivate all new notions in this work and prove quasimodularity of the algebra~$\SA$. A study of the symmetric algebra~$\TA$, including a proof of our main theorem can be found in \cref{sec:3}. The~$\sltwo$-action by differential operators, the proof of~\cref{thm:2}, and Rankin--Cohen brackets are the content of \cref{sec:sltwo}. In \cref{sec:2} further results that arise from comparing the two different products on~$\TA$ are given, and finally, in \cref{sec:application} we provide many examples of functions in or closely related to~$\TA$.  
%
%\clearpage
\numberwithin{thm}{paragraph}
\section{Preliminaries}\label{sec:pre}
\paragraph{Quasimodular forms}\label{par:QMF}
Let~$\mathrm{Hol}_0(\mathfrak{H})$ be the ring of holomorphic functions~$\phi$ of moderate growth on the complex upper half plane~$\mathfrak{H}$, i.e., for all~$C>0$ one has~$\phi(x+iy)=O(e^{Cy})$ as~$y\to \infty$ and~$\phi(x+iy)=O(e^{C/y})$ as~$y\to 0$. %Let~$\Gamma$ be a subgroup of~$\mathrm{SL}_2(\z)$ of finite index. %(or more generally a non-cocompact discrete subgroup of~$\mathrm{SL}_2(\r)$). 
%For~$\tau\in \mathfrak{H}$ and~$\gamma = \left(\begin{smallmatrix} a & b \\ c & d\end{smallmatrix}\right) \in \Gamma$, we write 
%\[(\phi|_k \gamma)(\tau) = (c\tau+d)^{-k}\phi\left(\frac{a\tau+b}{c\tau+d}\right).\]
A \textit{quasimodular form} of \textit{weight}~$k$ and \textit{depth} at most~$p$ for~$\sltwoz$ is a function~$\phi\in \mathrm{Hol}_0(\mathfrak{H})$ such that there exist~$\phi_0,\ldots, \phi_p\in \mathrm{Hol}_0(\mathfrak{H})$ so that for all~$\tau\in \mathfrak{H}$ and all~$\gamma = \left(\begin{smallmatrix} a & b \\ c & d\end{smallmatrix}\right) \in \sltwoz$ one has
\begin{align}\label{eq:modtrans}(c\tau+d)^{-k}\phi\Bigl(\frac{a\tau+b}{c\tau+d}\Bigr) \= \phi_0(\tau)\+\phi_1(\tau)\frac{c}{c\tau+d}\+\ldots\+\phi_p(\tau)\Bigl(\frac{c}{c\tau+d}\Bigr)^p.\end{align}
Equation~(\ref{eq:modtrans}) is called the \textit{quasimodular transformation property}. Note that if~$\phi$ is a quasimodular form, the functions~$\phi_0,\ldots, \phi_p$ are quasimodular forms uniquely determined by~$\phi$ (the function~$\phi_r$ has weight~$k-2r$ and depth~$\leq p-r$). For example, taking the identity~$I\in \Gamma$ yields~$\phi_0=\phi$. Quasimodular forms of depth~$0$ are called \emph{modular forms}. Besides the constant functions, the simplest examples are the \emph{Eisenstein series}
\[\label{def:eis}\hspace{40pt} G_k(\tau) \= -\frac{B_k}{2k}\+\sum_{r=1}^\infty\sum_{m=1}^\infty m^{k-1} q^{mr} \qquad (B_k = k\text{th Bernoulli number and } q=e^{2\pi i \tau}).\]
for positive even integers~$k$. For~$k>2$ the Eisenstein series are modular forms of weight~$k$. The Eisenstein series~$G_2$ is a quasimodular form of weight~$2$ and depth~$1$. 

Denote by~$\quasimodular_k^{(\leq p)}$ the vector space of quasimodular forms of weight~$k$ and depth at most~$p$. Often we omit the depth and/or weight and simply write~$\quasimodular_k$ for the vector space of all quasimodular forms of weight~$k$ or~$\quasimodular$ for the graded algebra of all quasimodular forms. Let~$\modular$ denote the graded algebra of modular forms. The quasimodular form~$G_2$ generates the algebra of quasimodular forms as an algebra over the subalgebra of modular forms, that is,~$\quasimodular=\modular[G_2]$. %In particular, every quasimodular form for~$\mathrm{SL}_2(\z)$ can be written as a polynomial in Eisenstein series:~$\quasimodular=\q[G_2,G_4,G_6]$. 

Often, when encountering an indexed collection of numbers or functions, we study its generating series. The generating series corresponding to the Eisenstein series is called the \emph{propagator} or the \emph{Kronecker--Eisenstein series of weight~$2$} and given by
\begin{equation}\label{eq:propagator}P(z,\tau)\=P(z)\defis \frac{1}{z^2} + 2\sum_{k=2}^\infty\Eis_k\frac{z^{k-2}}{(k-2)!} .\end{equation}
The propagator is closely related to the Weierstrass~$\wp$-function and Jacobi theta series
\begin{align}\label{eq:jacobitheta}
\wp(z,\tau)\,=\,\wp(z) \,:=\, \frac{1}{z^2}\,+\sum_{\substack{\omega \in \z \tau+ \z \\ \omega\neq 0}} \!\Bigl(\frac{1}{(z+\omega)^2}-\frac{1}{\omega^2}\Bigr),\qquad\theta(z)\,:=\,\sum_{\nu\in \z+{1\over 2}} (-1)^{\lfloor \nu\rfloor} e^{\nu z} q^{\nu^2/2}\end{align}
by
\[P(z)\=\frac{1}{2\pi i}\wp(\tfrac{z}{2\pi i},\tau)+2\Eis_2,\quad\quad P(z)\=-\pdv{}{z}\frac{\theta'(z)}{\theta(z)}.  \]
%
%~$\wp(z,\tau) = \frac{1}{z^2}+\sum_{\omega \in \z \tau+ \z}' \left(\frac{1}{(z+\omega)^2}-\frac{1}{\omega^2}\right)$, i.e.,
%Hence, similar to defining~$\Eis_2$ by choosing an ordering for the sum~$\sum_{\omega \in (2\pi i)(\z\tau+\z)}\frac{1}{\omega^2}$ one can define the propagator by choosing an ordering for the divergent sum 
% \[\sum_{\omega \in (2\pi i)(\z\tau+\z)}\frac{1}{(z+\omega)^2}.\qedhere\]

\paragraph{The action of~\texorpdfstring{$\sltwo$}{sl2} on quasimodular forms by derivations}
A way to produce examples of quasimodular forms is by taking derivatives of (quasi)modular forms under the differential operator~$D:\quasimodular_k^{(\leq p)} \to \quasimodular_{k+2}^{(\leq p+1)}$, given by
\[D\=\frac{1}{2\pi i} \frac{\mathrm{d}}{\mathrm{d}\tau} \= q\frac{\mathrm{d}}{\mathrm{d}q}.\]
In fact, every quasimodular form can uniquely be written as a linear combination of derivatives of modular forms and derivatives of~$G_2$. For more details, see \cite[p.~58--60]{Zag08}. It may happen that a polynomial in the derivatives of two modular forms~$f\in \modular_k$ and~$g\in \modular_l$ is actually modular. This is the case for the \emph{Rankin--Cohen brackets} of~$f$ and~$g$, defined by 
\[[f,g]_n \= \sum_{\substack{r,s\geq 0\\r+s=n}}(-1)^r\binom{k+n-1}{s}\binom{l+n-1}{r}D^r\!f\,D^s\!g \quad \quad (n\geq 0).\]
That is, for all~$f\in \modular_k, g\in \modular_l$ and~$n\geq 0$, one has that~$[f,g]_n$ is a modular form of weight~${k+l+2n}$. 

Besides the differential operator~$D$, an important differential operator on quasimodular forms is the operator~$\mathfrak{d}: \quasimodular_k^{(\leq p)} \to \quasimodular_{k-2}^{(\leq p-1)}$ defined by~$\phi\mapsto 2\pi i\phi_1$ (with~$\phi_1$ defined in the quasimodular transformation property~(\ref{eq:modtrans})). For example~${\mathfrak{d}G_2=-\frac{1}{2}}$ and in fact this property together with the fact that~$\mathfrak{d}$ annihilates modular forms defines~$\mathfrak{d}$ completely since~$\mathfrak{d}$ is a derivation and~${\quasimodular=\modular[G_2]}$.

 Let~$W$ be the weight operator, which multiplies a quasimodular form by its weight. The triple~$(D, \mathfrak{d}, W)$ forms an~\emph{$\sltwo$-triple} with respect to the commutator bracket~${[A,B]=AB-BA}$:
\begin{defn} A triple~$(X,Y,H)$ of operators is called an~\emph{$\sltwo$-triple} if
\[[H,X]\=2X, \quad [H,Y]\=-2Y, \quad [Y,X]\=H.\]
\end{defn}
\begin{remark} By these commutation relations, for all~$n\geq 1$ one has
\begin{equation}\label{eq:D^n} [\mathfrak{d},D^n] \=  n(W-n+1)D^{n-1},\end{equation}
which turns out to be useful later. 
\end{remark}

Following a suggestion of Zagier, we make the following definition:
\begin{defn}\label{def:g-alg}
Given a Lie algebra~$\mathfrak{g}$, a~$\mathfrak{g}$-algebra is an algebra~$A$ together with a Lie homomorphism~$\mathfrak{g}\to \mathrm{Der}(A).$
\end{defn}
As~$D,\mathfrak{d}$ and~$W$ satisfy the Leibniz rule, the algebra~$\widetilde{M}$ becomes an~$\sltwo$-algebra.

\paragraph{Partitions as a partially ordered set}
Given~$n\in \z_{\geq 0}$, let~$\partitions(n)$ denote the set of all integer partitions of~$n$ and~$\Pi(n)$ the set of all partitions of the set~$[n]:=\{1,2,\ldots, n\}$. Let~$\partitions=\bigcup_{n\in \z_{\geq 0}} \partitions(n)$ and~$\Pi=\bigcup_{n\in \z_{\geq 0}} \Pi(n)$ be the sets of all such partitions. Given~$\lambda \in \partitions(n)$ we write~$\lambda=(\lambda_1,\lambda_2,\ldots)$ with~$\lambda_1\geq \lambda_2\geq \ldots$ and~$|\lambda|:=\sum_{i=1}^\infty \lambda_i=n$. The largest index~$k$ such that~$\lambda_k>0$ is called the length of~$\lambda$, denoted by~$\ell(\lambda)$. Similarly, for~$\alpha\in\Pi(n)$ we write~$\ell(\alpha)$ for the cardinality of~$\alpha$. %and use~$|\alpha|$ to denote the integer~$n$. 
Moreover, for~$\lambda\in\partitions$ we let~$r_m(\lambda)$ denote the number of parts of~$\lambda$ equal to~$m$, i.e.~$r_m(\lambda)=\#\{i\mid \lambda_i=m\}$, and denote by~$\lambda'$ the conjugate partition of~$\lambda$. We call a partition~$\lambda$ \emph{strict} if there are no repeated parts, i.e.~$r_m(\lambda)\in \{0,1\}$ for all~$m$. For two partitions~$\kappa,\lambda$ we write~$\kappa\cup \lambda$ for the union of~$\kappa$ and~$\lambda$ as multisets, i.e.~${r_m(\kappa\cup\lambda)=r_m(\kappa)+r_m(\lambda)}$ for all~$m\in \n$. 

Both~$\partitions$ and~$\Pi(n)$ form a locally finite partially ordered set, i.e., a partially ordered set~$P$ for which for all~$x,z\in P$ there exists finitely many~$y\in P$ such that~$x\leq y\leq z$. Namely, on~$\partitions$ we define a partial order by~$\kappa\leq \lambda$ if~$r_m(\kappa)\leq r_m(\lambda)$ for all~$m\geq 1$. The ordering on~$\Pi(n)$ is given by~$\alpha\leq \beta$ if for all~$A\in \alpha$ there exists a~$B\in \beta$ such that~$A\subseteq B$. For instance, we have~$\alpha\leq \one$ for all~$\alpha\in\Pi(n)$, where~$\one=\{[n]\}$.

Recall that on a locally finite partially ordered set~$P$ the M\"obius function $\mu:P^2\to\z$ is defined recursively by (see for example \cite{R64}):
$\mu(x,z) \=  -\sum_{x\leq y\leq z} \mu(x,y)$ if ~$x<z$ with initial conditions~$\mu(x,x)=1$ and~$\mu(x,z)=0$ else. For the above partial order on~$\mathscr{P}$ the value of~$\mu(\kappa,\lambda)$ depends on whether the difference of~$\kappa$ and~$\lambda$ considered as multisets, denoted by~$\lambda- \kappa$, is a strict partition. That is,
\begin{equation}\label{eq:partitionmobius}\mu(\kappa,\lambda) \= \begin{cases}(-1)^{\ell(\lambda)-\ell(\kappa)} & \lambda- \kappa \text{ is a strict partition} \\ 0 & \text{else.}\end{cases}\end{equation}
The M\"obius function~$\mu(\alpha,\beta)$ of two elements $\alpha,\beta\in \Pi(n)$ is given by 
\[\mu(\alpha,\beta)\=\prod_{B\in\beta}(-1)^{\ell(\alpha_B)-1}(\ell(\alpha_B)-1)!\, ,\]
 where~$\alpha_B$ for $B\subset [n]$ is the partition on~$B$ induced by~$\alpha$. 
A M\"obius function satisfies the following two properties:
\begin{thm}\label{thm:mobius}
Let~$f,g$ be functions on a partially ordered set~$P$. Then
\begin{enumerate}[\upshape(i)]
\item\label{it:rdm} $\displaystyle\sum_{\alpha\leq \gamma\leq \beta} \mu(\alpha,\gamma) \= \delta_{\alpha,\beta} \= \sum_{\alpha\leq \gamma\leq \beta} \mu(\gamma,\beta)$\quad for all $\alpha,\beta\in P$;  %\hfill (Recursive definition of $\mu(\alpha,\beta)$)
\item\label{it:miv} $\displaystyle f(\alpha)=\sum_{\gamma\leq\alpha}g(\gamma) \quad \forall \alpha \in P \quad \iff \quad g(\beta)=\sum_{\gamma\leq\beta}\mu(\gamma,\beta)\,f(\gamma) \quad \forall \beta \in P.$ %\\ \mbox{}\hfill (M\"obius inversion formula)
\end{enumerate}
\end{thm}

\paragraph{The connected~\texorpdfstring{$q$}{q}-bracket}
The~$q$-bracket defined in the introduction (Eqn~(\ref{eq:qbrac})) is a map~$\q^{\partitions}\to \q[[q]]$. In this section we define the connected~$q$-bracket following \cite[p.~55--57]{CMZ16}, which naturally arises in enumerative geometric when counting \textit{connected} coverings. In our setting, the connected~$q$-bracket turns out to be easier to compute than the usual~$q$-bracket.

For $A\subset [n]$ we denote $f_A = \prod_{a\in A} f_a$. 
\begin{defn}\label{def:qbrac} Given an integer $n\geq 1$, the \emph{connected~$q$-bracket} is defined as the multilinear map \[\langle\ \rangle_q:\underbrace{\QtoP \otimes\cdots \otimes \QtoP}_{n}\to \q\] extending the~$q$-bracket such that for all $f,f_1,\ldots, f_n\in \QtoP$ any of the following two equivalent conditions hold:
\begin{enumerate}[\upshape (i)]
%\item \label{eq:conbraci}
%$\begin{aligned}[t]
%\langle f\otimes f_1 \otimes &\cdots \otimes f_n\rangle_q \=\! \\
%& \sum_{A\sqcup B = [n]} (-1)^{|A|-1}\left(\Big\langle f f_A \otimes \bigotimes_{b\in B} f_b \Big\rangle_q - \langle f \rangle_q \Big\langle f_A \otimes \bigotimes_{b\in B} f_b \Big\rangle_q\right);\end{aligned}$
\item \label{eq:conbrac}$\displaystyle \langle f_1 \otimes \cdots \otimes f_n\rangle_q \= \sum_{\alpha \in \Pi(n)} \mu(\alpha,\one)\prod_{A\in \alpha}\langle f_A\rangle_q\,;$
\item \label{eq:conbraciii}$\langle f_1\otimes \cdots \otimes f_n\rangle_q$ is the coefficient of~$x_1\cdots x_n$ in~$\log\langle \exp(\sum_{i=1}^n x_i f_i)\rangle_q\, .$
\end{enumerate}
\end{defn}
By invoking the M\"obius inversion formula (Theorem~\ref{thm:mobius}(\ref{it:miv})) condition~(\ref{eq:conbrac}) in \cref{def:qbrac} implies that
\begin{align}\label{eq:connectedbracket}\prod_{B\in \beta}\langle \otimes_{b\in B}f_b \rangle_q &\= \sum_{\alpha\leq\beta} \mu(\alpha,\beta)\prod_{A\in \alpha}\langle f_A\rangle_q\, , \qquad
\prod_{A\in \alpha}\langle f_A\rangle_q\= \sum_{\beta\leq\alpha} \prod_{B\in \beta}\langle \otimes_{b\in B}f_b \rangle_q \,.
\end{align}
For example,
\begin{align}
\langle f\otimes g\rangle_q &\= \langle fg\rangle_q-\langle f\rangle_q\langle g\rangle_q\, , \\
\langle f\otimes g\otimes h\rangle_q & \= \langle fgh\rangle_q-\langle f\rangle_q\langle gh\rangle_q-\langle g\rangle_q\langle fh\rangle_q-\langle h\rangle_q\langle fg\rangle_q+2\langle f\rangle_q\langle g\rangle_q\langle h\rangle_q\, , 
\end{align}
and
\begin{align}
\langle fg\rangle_q &\= \langle f\otimes g\rangle_q+\langle f\rangle_q\langle g\rangle_q \, ,  \\
\langle fgh \rangle_q & \= \langle f\otimes g\otimes h\rangle_q+\langle f\rangle_q\langle g\otimes h\rangle_q+\langle g\rangle_q\langle f \otimes h\rangle_q+\langle h\rangle_q\langle f \otimes g\rangle_q+\langle f\rangle_q \langle g\rangle_q \langle h\rangle_q\, . 
\end{align}

We often make use of the fact that the connected~$q$-bracket of functions~$f_1,\ldots, f_n$  vanishes if one of the~$f_i$ is constant.
\begin{lem}\label{lem:constantconnected}
For all $f_1,\ldots,f_n\in \QtoP$ one has
\[\langle 1 \otimes f_1\otimes \cdots \otimes f_n \rangle_q = 0.\]
\end{lem}
\begin{proof} Write $f_{n+1}=1$.
Observe that~$\prod_{A\in \alpha}\langle f_A\rangle_q$ takes the same value for all ${\alpha\in\Pi(n+1)}$ which agree on~$[n]$ (but differ in the subset~$A$ of~$\alpha$ containing~${n+1}$).  Then, summing~$\mu(\alpha,\one)$ over all such~$\alpha$ yields
\[a\cdot (-1)^{a-1}(a-1)! + (-1)^{a}a! = 0\]
as there are~$a$ choices for~$\alpha$ for which~$\{n+1\}$ is not a subset of $\alpha$, where~$a$ is the length of such an~$\alpha$, and there is only one choice for~$\alpha$ for which~$\{n+1\}$ is a subset. By \cref{def:qbrac}(\ref{eq:conbrac}) the result follows. 
\end{proof}
We will use the second condition in \cref{def:qbrac} in our proof that~$\SA$ is a quasimodular algebra.

\paragraph{The discrete convolution product and Faulhaber polynomials}
Let~$\n$ denote the set of strictly positive integers. Given $f,g:\n\to \q$ we denote by~$f\cdot g$ or~$fg$ the pointwise product of~$f$ and~$g$. We define the \emph{discrete convolution product} of~$f$ and~$g$ by
\[(f*g)(n)\=\sum_{i=1}^{n-1} f(i)\,g(n-i)\]
and denote the convolution product of functions~$f_1,\ldots, f_n$ by
\begin{equation}\label{eq:conv}\bigconv_{i=1}^n f_i \= f_1 * \cdots * f_n.\end{equation}
Let the \emph{discrete derivative}~$\partial$ of $f:\n\to \q$ be defined by $\partial f(n) = f(n)-f(n-1)$ for $n\geq 2$ and $\partial f(1)=f(1)$ and denote by~$\mathrm{id}$ the identity function $\n\to\n\subset\q$. Observe that
\begin{align}\label{convleibniz}
\partial (f*g) &\= (\partial f)*g\=f*(\partial g), \\
\partial(fg) &\= \partial(f)\,g+f\,\partial(g)-\partial(f)\,\partial(g), \label{prodleibniz}\\
\mathrm{id}\cdot(f*g) &\=(\mathrm{id}\cdot f)*g+f*(\mathrm{id}\cdot g), \label{eq:id(f*g)}\\
\partial^2(f*\mathrm{id}) &\=f-\partial f. \label{eq:f*id}
\end{align}
The \emph{Faulhaber polynomials}~$\Faulhaber_l$ for $l\ge1$ are defined as the unique polynomials with vanishing constant term satisfying $\partial\Faulhaber_l(n)=n^{l-1}$ for all $n\in \n$, or equivalently by $\Faulhaber_l(n)=\sum_{i=1}^n i^{l-1}$. The first four are given by
\[\Faulhaber_1(x)=x, \quad \Faulhaber_2(x)=\frac{x(x+1)}{2}, \quad \Faulhaber_3(x) = \frac{x(x+1)(2x+1)}{6}, \quad \Faulhaber_4(x)=\frac{x^2(x+1)^2}{4}.\]
Note that these polynomials are related to the Bernoulli polynomials~$B_n(x)$, the unique family of polynomials satisfying $\int_{x}^{x+1}B_n(u)\,\mathrm{d}u = x^n$, by the formula~$l\Faulhaber_l(x) = B_{l}(x+1)-B_{l}$. Hence, the Faulhaber polynomials admit the symmetry
\begin{equation}\label{eq:faulhabersym}\Faulhaber_l(x) \= (-1)^l\Faulhaber_l(-x-1) \quad\quad (l \geq 2),\end{equation}
which can also be deduced directly from the definition.
The generating series~$\mathpzc{F}\!(n)$ of the Faulhaber polynomials equals
\begin{equation}\label{def:gsfaulhaber}
\mathpzc{F}\!(n)\defis\sum_{l=1}^\infty \Faulhaber_l(n)\frac{z^{l-1}}{(l-1)!} \= e^z\frac{1-e^{nz}}{1-e^z}.
\end{equation}

\section{The moment functions, their~\texorpdfstring{$q$}{q}-bracket and a second product}\label{sec:defn}
\paragraph{Three proofs of the quasimodularity of the moment functions}\label{sec:3proofs}
The~$q$-bracket of the moment function~$S_k$ defined in~(\ref{eq:defSk}) equals the Eisenstein series~$G_k$. To motivate the results in the rest of this work, we provide three different proofs---and three generalizations---of this statement using three different approaches. In the first approach we motivate the definition of the~$T_{k,l}$ (see~(\ref{def:Tkl0})), the second approach gives an interpretation for these functions, and the last approach gives an example of our main principle of establishing all identities before taking the~$q$-bracket. 

\subparagraph{First approach} 
The key observation in this first proof is that~$S_k$ can be rewritten as
\begin{align} \label{def:sk1}
S_k(\lambda)\= -\frac{B_k}{2k}+\sum_{m=1}^\infty m^{k-1} r_m(\lambda).
\end{align}
More generally, for $k> 0$ and $f:\n\to \q$ we set $f(0)=0$ and we let
\begin{equation}\label{def:Tkf}
S_{k,f}(\lambda) \= -\frac{B_{k+1}}{2(k+1)}\delta_{f,\mathrm{id}} \+\sum_{m=1}^\infty m^{k} f(r_m(\lambda)).
\end{equation}
In case when~$f$ is the identity, $S_{k,f}=S_{k+1}$. Our first method of proof gives the following more general statement:
\begin{prop}\label{prop:faulhaber}
Let~$f$ be a polynomial of degree~$l$ without constant term and~$k$ a positive integer satisfying $k\equiv l \bmod 2$. Then, 
\begin{enumerate}[\upshape (i)]
\item if~$f$ equals a Faulhaber polynomial~$\Faulhaber_l$, then~$\langle S_{k,f}\rangle_q$ equals
\[ -\frac{B_{k+1}}{2(k+1)}\delta_{l,1}+\sum_{m,r\geq 1} m^{k}r^{l-1} q^{mr} \= \begin{cases} D^{l-1}G_{k-l+2} & k-l\geq 0, \\ D^k G_{l-k} & k-l\leq 2;\end{cases}\]
\item\label{it:prop2} if~$\langle S_{k,f}\rangle_q$ is a quasimodular form, then~$f$ is a multiple of the Faulhaber polynomial~$\Faulhaber_l$. 
\end{enumerate}
\end{prop}
\begin{proof}
Let $|x|\leq 1$ and $m\geq 1$. We compute
\begin{align}\label{eq:qbracxrm}
\langle x^{r_m}\rangle_q \= \frac{\sum_{\lambda \in \partitions} x^{r_m(\lambda)}q^{|\lambda|}}{\sum_{\lambda \in \partitions} q^{|\lambda|}}.
\end{align}
Observe that the multiplicities~$r_1(\lambda), r_2(\lambda), \ldots$ uniquely determine the partition~$\lambda$. Hence, for $|q|<1$ we have that
\begin{align}
\sum_{\lambda \in \partitions} x^{r_m(\lambda)}q^{|\lambda|} &\= \sum_{r_1,r_2,\ldots \geq 0} x^{r_m} q^{r_1+2r_2+\ldots + m r_m + \ldots} \\
&\= \left(\sum_{r_m=0}^\infty x^{r_m}q^{m r_m} \right)\prod_{i\neq m} \left(\sum_{r_i=0}^\infty q^{i r_i}\right) \\
&\= \frac{1}{1-x q^m} \prod_{i\neq m} \frac{1}{1-q^i}\, .
\end{align}
Substituting this result in the numerator of~(\ref{eq:qbracxrm}), we obtain
\begin{align}
\langle x^{r_m}\rangle_q \= \frac{1-q^m}{1-xq^m}\, . 
\end{align}

Hence,
\begin{align}\label{eq:qbracsumrm}\left\langle \tfrac{x}{1-x}(1-x^{r_m})\right\rangle_q \= \frac{xq^m}{1-xq^m}\, .\end{align}
Observe that applying~$x\pdv{}{x}$ to the right-hand side of~(\ref{eq:qbracsumrm}) has the same effect as applying~$\frac{1}{m}D$, where~$D$ is defined in §\ref{par:QMF}. 
After setting $x=e^z$, we find that~$\frac{x}{1-x}(1-x^{r_m})$ equals~$\mathpzc{F}\!(r_m)$ (see~(\ref{def:gsfaulhaber})). Hence, by taking~${l-1}$ derivatives~$x\pdv{}{x}=\pdv{}{z}$  and setting $z=0$, it follows that 
\begin{align}\langle S_{k,\Faulhaber_l} \rangle_q +\frac{B_{k+1}}{2(k+1)}\delta_{l,1} &\= \sum_{m\geq 0} m^{k} \langle \Faulhaber_l(r_m)\rangle_q \\
&\=\sum_{m\geq 0} m^k \left(x\pdv{}{x}\right)^{l-1} \frac{xq^m}{1-xq^m}\Big|_{x=1} \\
&\=\sum_{m\geq 0} m^{k} \left(\frac{1}{m}D\right)^{l-1} \frac{q^m}{1-q^m} \\
&\=\sum_{m,r\geq 1} m^{k}r^{l-1} q^{mr}.
%\= \begin{cases} D^{l-1}G_{k-l+2} & k-l\geq 0 \\ D^k G_{l-k} & k-l\leq 2;\end{cases}.
\end{align}
Part~(\ref{it:prop2}) of the statement follows by writing~$f$ as a linear combination of Faulhaber polynomials. 
\end{proof}

\subparagraph{Second approach} 
The \emph{double moment functions}~$T_{k,l}$ (see~(\ref{def:Tkl0})) are by definition equal to~$S_{k,\Faulhaber_l}$ if $k>0$. Given a partition~$\lambda$, let $c_i(\lambda) = \#\{j\leq i \mid \lambda_i=\lambda_j\}$. Then, one has
\begin{equation}\label{def:Tkl}T_{k,l}(\lambda) \= -\frac{B_{k+l}}{2(k+l)}(\delta_{l,1}+\delta_{k,0}) + \sum_{i=1}^\infty \lambda_i^k c_i(\lambda)^{l-1}.\end{equation}
In this section we give a direct proof for the quasimodularity of the~$q$-brackets of~$T_{k,l}$:

\begin{prop}\label{prop:ap2} For all $k\geq 0,l\geq 1$ and~${k+l}$ even, one has
\[\langle T_{k,l}\rangle_q \= \begin{cases} D^{l-1}G_{k-l+2} & \text{if } k-l\geq 0, \\ D^k G_{l-k} & \text{if } k-l\leq 2.\end{cases}\]
\end{prop}
\begin{proof}
Denote by $T_{k,l}^0(\lambda) = \sum_{i=1}^\infty \lambda_i^k c_i(\lambda)^{l-1}$. 
The generating series of~$T_{k,l}^0$ is given by
\[W(X,Y)(\lambda) \= \sum_{i=1}^\infty X^{\lambda_i}Y^{c_i(\lambda)},\]
that is,~$T_{k,l}^0(\lambda)$ is the coefficient of~$\frac{x^ky^{l-1}}{k!(l-1)!}$ in~$W(e^x,e^y)(\lambda)$. 
Consider 
\begin{equation}\label{eq:Wtilde}
\sum_{\lambda \in \partitions} W(X,Y)(\lambda)\, q^{|\lambda|}\=\sum_{\lambda \in \partitions}\sum_{i=1}^\infty X^{\lambda_i}Y^{c_i(\lambda)}q^{|\lambda|}.\end{equation}
Given $a,b,n\in \z_{\geq 0}$, denote by~$C_{a,b}(n)$ the coefficient in front of~$X^{a}Y^bq^{n}$ in~(\ref{eq:Wtilde}), that is
\[\sum_{\lambda \in \partitions}\sum_{i=1}^\infty X^{\lambda_i}Y^{c_i(\lambda)}q^{|\lambda|} \;=:\; \sum_{a,b,n\geq 0} C_{a,b}(n) X^{a}Y^{b}q^{n}.\]
Let~$p(n)$ denote the number of partitions of~$n$. The coefficient~$C_{a,b}(n)$ equals the number of partitions of~$n$ with at least~$b$ parts of size~$a$, i.e., $C_{a,b}(n)=p(n-ab)$. Hence, writing $m=n-ab$ we obtain
\[\sum_{\lambda \in \partitions}\sum_{i=1}^\infty X^{\lambda_i}Y^{c_i(\lambda)}q^{|\lambda|} \= \left(\sum_{m=0}^\infty p(m)q^{m}\right)\left(\sum_{a,b\geq 0}X^{a}Y^{b} q^{ab}\right).\]
In other words,
\[\langle W(X,Y)\rangle_q \=\sum_{a,b\geq 0} X^{a}Y^{b}q^{ab},\]
so that expanding this equation for $X=e^x$ and $Y=e^y$ yields 
\[\langle T_{k,l}^0 \rangle_q \= \sum_{a,b\geq 0}a^kb^{l-1} q^{ab}.\]
As $T_{k,l}(\lambda) = -\frac{B_{k+l}}{(k+l)}(\delta_{l,1}+\delta_{k,0}) + T_{k,l}^0(\lambda)$ we obtain the desired result. 
\end{proof}

\subparagraph{Third approach} In this last proof we start with the observation that one can rewrite the~$q$-bracket as
\begin{equation}\label{eq:qbrac2}\langle f\rangle_q \= \frac{\sum_{\lambda\in \partitions} f(\lambda)\, u_{\lambda_1} u_{\lambda_2} \cdots }{\sum_{\lambda \in \partitions} u_{\lambda_1} u_{\lambda_2} \cdots}\Big|_{u_i=q^{i}}\, . \end{equation}
In contrast to the previous two proofs, it is only in the last step of this proof that we take the~$q$-bracket: first we rewrite~(\ref{eq:qbrac2}) considering~$u_1,u_2,\ldots$ to be formal variables, and in the last step we let $u_i=q^i$. We start with the denominator, where we encounter the M\"obius function on partitions also defined in \cite{Sch16}.

\begin{prop}\label{prop:mobius} There exists a function $\mu:\partitions \to \{-1,0,1\}$ defined by any one of the following three equivalent definitions: \begin{enumerate}[\upshape (i)]
\item $\mu(\lambda)$ is given by the M\"obius function~$\mu(\emptyset, \lambda)$ on the partial order on the set of partitions in \upshape(\ref{eq:partitionmobius});
\item $\mu(\lambda) \= \begin{cases} (-1)^{\ell(\lambda)} & \lambda \text{ is a strict partition} \\ 0 & \text{else;}\end{cases}$
\item $\displaystyle \frac{1}{\sum_{\lambda \in \partitions} u_{\lambda_1} u_{\lambda_2} \cdots} \= \sum_{\lambda \in \partitions} \mu(\lambda)\, u_{\lambda_1} u_{\lambda_2} \cdots.$
\end{enumerate}
\end{prop}
\begin{proof}
The first two definitions clearly coincide using \upshape(\ref{eq:partitionmobius}). For the latter, it suffices to show that
\[\sum_{\alpha\cup\beta=\lambda} \mu(\alpha) \= \delta_{\lambda,\emptyset}\, .\]
Let $f(\lambda)=1$ and $g(\lambda)=\delta_{\lambda,\emptyset}$ for $\lambda \in \partitions$. Then, $f(\alpha)=\sum_{\gamma\leq \alpha}g(\gamma)$ for all $\alpha\in \partitions$, so that by M\"obius inversion and by using $\mu(\gamma,\beta) = \mu(\emptyset,\beta-\gamma)$ the last definition is equivalent. 
\end{proof}

The fact that $\langle S_k\rangle_q=G_k$ follows directly from the following proposition:
\begin{prop}\label{prop:phibrac} For all $m\geq 1$ and $f:\n\to \q$ extended by $f(0)=0$, one has
\[\frac{\sum_{\lambda \in \partitions} f(r_m(\lambda))\, u_{\lambda_1} u_{\lambda_2} \cdots}{\sum_{\lambda \in \partitions} u_{\lambda_1}u_{\lambda_2}\cdots} \= \sum_{r=1}^\infty \partial f(r)\, u_m^r.\]
\end{prop}
\begin{proof}
Fix $m\geq 1$. By the previous proposition, we have 
\[\frac{\sum_{\lambda \in \partitions} f(r_m(\lambda))\, u_{\lambda_1} u_{\lambda_2} \cdots}{\sum_{\lambda \in \partitions} u_{\lambda_1}u_{\lambda_2}\cdots} \= \left(\sum_{\lambda \in \partitions} f(r_m(\lambda))\, u_{\lambda_1} u_{\lambda_2} \cdots\right)\left(\sum_{\lambda \in \partitions} \mu(\lambda)\, u_{\lambda_1} u_{\lambda_2} \cdots\right).\]
Denote by~$C(\lambda)$ the coefficient of~$u_{\lambda_1}u_{\lambda_2}\cdots$ after expanding the right-hand side of above equation. Observe that
\[C(\lambda) \= \sum_{\alpha\cup\beta=\lambda} (-1)^{\ell(\beta)} f(r_m(\alpha)),\]
where~$\alpha\cup \beta$ denotes the union of~$\alpha$ and~$\beta$ considered as multisets and it is understood that~$\beta$ is a strict partition. Suppose~$\lambda$ admits a part equal to $m'\neq m$. Then, define an involution~$\omega$ on all pairs~$(\alpha,\beta)$ satisfying that $\alpha\cup\beta=\lambda$ and~$\beta$ is strict by 
\[\omega(\alpha,\beta) \= 
\begin{cases} 
(\alpha\backslash\{m'\},\beta\cup \{m'\}) & \text{if } r_{m'}(\beta)=0, \\ 
(\alpha\cup\{m'\},\beta\backslash \{m'\}) & \text{if } r_{m'}(\beta)=1.
\end{cases}\]
As~$\omega$ changes the sign of~$(-1)^{\ell(\beta)}f(r_m(\alpha))$, it follows that $C(\lambda)=0$. 

Observe that $C(\emptyset)=0$ and that in case $\lambda=(m,m,\ldots)$ consists of a strictly positive number of parts all equal to~$m$ one has
\[C(\lambda) \= f(r_m(\lambda))-f(r_m(\lambda)-1)\=\partial f(r_m(\lambda)).\]
Therefore, the desired result follows. 
\end{proof}

\paragraph{The induced and connected product}\label{par:3.2}
Motivated by the last of the three approaches in the previous section, we define the~\emph{$\vec{u}$-bracket} of a function $f\in \q^\partitions$ by 	
\[\langle f\rangle_{\vec{u}} \= \frac{\sum_{\lambda\in \partitions} f(\lambda)\, u_\lambda }{\sum_{\lambda \in \partitions} u_\lambda} \quad\quad\quad (u_\lambda=u_{\lambda_1} u_{\lambda_2}\cdots).\]
Then, for all $f\in \QtoP$ one has~$\langle f \rangle_q=\langle f \rangle_{(q, q^2, q^3,\ldots)}$.
Observe that the~$\vec{u}$-bracket defines an isomorphism of vector spaces 
\[\QtoP \xrightarrow{\sim} \q[[u_1,u_2,u_3,\ldots]], \quad \quad f \mapsto \langle f \rangle_{\vec{u}}\, .\]
 We now use the algebra structure of~$\q[[u_1,u_2,u_3,\ldots]]$ to define a product on~$\QtoP$. 
\begin{defn}\label{def:inducedprod} Given $f,g\in \QtoP$ we define their \emph{induced product}~${f\blok g}$ by
\[\langle f\blok g \rangle_{\vec{u}}\= \langle f\rangle_{\vec{u}}\langle g\rangle_{\vec{u}}\, ,\]
where the product of~$\langle f\rangle_{\vec{u}}$ and~$\langle g\rangle_{\vec{u}}$ is the usual product of power series.
\end{defn}
\begin{remark} Observe that~$\QtoP$ is a commutative algebra with the constant function~$1$ as the identity for both the pointwise and the induced product. This observation should be compared with the~$q$-bracket arithmetic in \cite{Sch16}.
\end{remark}
The following proposition gives an alternative definition for the induced product. 
\begin{prop}\label{prop:defblok2} For all~$\lambda \in \partitions$, one has
\[(f\blok g)(\lambda) \;= \sum_{\alpha\cup\beta\cup\gamma=\lambda} f(\alpha)\,g(\beta)\,\mu(\gamma).\]
\end{prop}
\begin{proof}
By definition
\[\sum_{\lambda\in \partitions}(f\blok g)(\lambda)\,u_\lambda \= \frac{\left(\sum_{\lambda\in \partitions} f(\lambda)\, u_\lambda \right)\left(\sum_{\lambda\in \partitions} g(\lambda)\, u_\lambda \right)}{\sum_{\lambda \in \partitions} u_\lambda}.\]
By \cref{prop:mobius} this equals 
\[\left(\sum_{\lambda\in \partitions} f(\lambda)\, u_\lambda \right)\left(\sum_{\lambda\in \partitions} g(\lambda)\, u_\lambda \right)\left(\sum_{\lambda\in \partitions} \mu(\lambda)\, u_\lambda \right).\]
The result follows by expanding the products. 
\end{proof}

Analogous to the connected~$q$-bracket, we define the connected product. For a set~$S$ and functions $f_s\in \QtoP$ for all $s\in S$, we denote $f_S = \prod_{s\in S}f_s\, .$
\begin{defn}For~$f_1,\ldots, f_n\in \QtoP$, define the \emph{connected product}~$f_1\,|\,\ldots| f_n$ to be the following function~$\partitions \to \q$:
\begin{align}\label{def:connectedprod}f_1\cp\ldots\cp f_n\defis\sum_{\alpha\in\Pi(n)} \mob \bigblok_{A\in\alpha} f_A\, .\end{align}
\end{defn}
For example, for~$f,g,h\in\QtoP$ one has
\begin{align}
f\cp g &\=  fg- f\blok g, \\
 f\cp g\cp h &\=  fgh- f \blok gh- g \blok fh- h \blok fg+2 f\blok g \blok h.
\end{align}

The induced and connected product allow us to establish many identities before taking the~$q$-bracket, as follows from the following result. 
\begin{prop}\label{prop:prod} For all~$f_1,\ldots, f_n\in \QtoP$ one has
\begin{itemize}
\item $\displaystyle \langle f_1\blok f_2 \blok \cdots \blok f_n\rangle_q \= \langle f_1\rangle_q \langle f_2\rangle_q \cdots \langle f_n\rangle_q\, $;
\item $\displaystyle \langle f_1\cp \ldots\cp  f_n\rangle_q \= \langle f_1 \otimes \cdots \otimes f_n\rangle_q\,$. 
\end{itemize}
\end{prop}
\begin{proof}
\sloppy
Both statements follow directly from the definitions. For the first, note that for all~${f,g\in \QtoP}$ one has
\[\langle f\blok g\rangle_q \= \langle f\rangle_{\vec{u}}\langle g\rangle_{\vec{u}}|_{u_i=q^i} \= \langle f\rangle_{\vec{u}}|_{u_i=q^i}\langle g\rangle_{\vec{u}}|_{u_i=q^i} \= \langle f\rangle_q\langle g\rangle_q,\]
so that the statement follows inductively. The second follows from the first, as
\begin{align}\displaystyle \langle f_1\cp \ldots\cp  f_n\rangle_q &\= \sum_{\alpha\in\Pi(n)} \mob \prod_{A\in\alpha} \langle f_A\rangle_q \= \langle f_1 \otimes \cdots \otimes f_n\rangle_q\, .\qedhere\end{align}
\end{proof}
\begin{remark}\label{rk:r} Let~$\mathcal{R}$ be the space of functions having a quasimodular form as~$q$-bracket, i.e.~${\mathcal{R}=\langle\, \cdot\,\rangle_q^{-1}(\widetilde{M})}$. Then,~$\mathcal{R}$ is a graded algebra with multiplication given by the induced product. Namely, if~$f\in \mathcal{R}$ and~$\langle f \rangle_q\in \quasimodular_k$, we define the weight of~$f$ to be equal to~$k$. Note that if~$f,g\in \mathcal{R}$ and~$\langle f \rangle_q$ and~$\langle g \rangle_q$ are quasimodular forms of weight~$k$ and~$l$ respectively, then~$\langle f \blok g\rangle_q = \langle f\rangle_q\langle g\rangle_q$ is a quasimodular form of weight~$k+l$. 
\end{remark}

When establishing identities on the level of functions on partitions (before taking the $q$-bracket), it turns out to be very useful to express the connected product of pointwise products of elements of~$\QtoP$ in terms of connected and induced products. This can be done recursively using the following result. %, where for a set~$S$ and~$f\in \QtoP$ we denote~$f_S = \prod_{s\in S}f_s$.  
\begin{prop}\label{prop:connectedbracketrecursion}
For all~$f_1,\ldots f_n\in \QtoP$ one has
\begin{align}\label{eq:recursion}f_1f_2\cp f_3\cp f_4\cp \ldots\cp  f_n \=  f_1&\cp f_2\cp \ldots\cp  f_n \; \+\\
&\sum_{A\sqcup B=\{3,\ldots, n\}} (f_1\cp f_{A_1}\cp f_{A_2}\cp \ldots)\blok(f_2\cp f_{B_1}\cp f_{B_2}\cp \ldots),\end{align}
where $A_1,A_2,\ldots$ enumerate the elements of~$A$ (and similarly for~$B$). 
\end{prop}
\begin{proof}
Observe that both sides of the equation in the statement are a linear combination of terms of the form~$\bigblok_{C\in \gamma}f_C$ over~$\gamma\in \Pi(n)$. We determine the coefficient of such a term on both sides of the equation.

First of all, assume~$\gamma$ is such that~$\{1,2\}\subset C$ for some~$C\in \gamma$. Then, on the right hand side such a term only occurs in~$f_1\cp \ldots\cp  f_n$ with coefficient~$\mu(\gamma,\mathbf{1})$. Moreover, let~$\tilde\gamma\in \Pi(n-1)$ be given by~$\gamma\cap {\{2,\ldots,n\}}$ subject to replacing~$i$ by~$i-1$ for all~$i=2,\ldots, n$. Note that the coefficient on the left-hand side equals~$\mu(\tilde\gamma,\mathbf{1})$. As~$\ell(\tilde\gamma)=\ell(\gamma)$, the coefficients on both sides agree. 

Next, assume~$C_1,C_2\in \gamma$ with~$1\in C_1$ and~$2\in C_2\,$. Then, the coefficient of~$\bigblok_{C\in \gamma}f_C$ on right-hand side of~(\ref{eq:recursion}) equals
\begin{align}\label{eq:doublemobius}\mu(\gamma,\mathbf{1})+\sum \mu(\gamma|_{A},\mathbf{1})\mu(\gamma|_B,\mathbf{1}),\end{align}
where the sum is over all~$I\subset \{2,3,\ldots,\ell(\gamma)\}$ and~$A$ and~$B$ are given by~$A=C_1\cup\bigcup_{i\in I}C_i$ and~$B=C_2\cup\bigcup_{i\in I^c}C_i$. Letting~$i$ be the number of elements of~$I$, we find that~(\ref{eq:doublemobius}) equals
\begin{align}&\mu(\gamma,\mathbf{1})+\sum_{i=0}^{\ell(\gamma)-2} \binom{\ell(\gamma)-2}{i}\cdot(-1)^{i}i!\cdot(-1)^{\ell(\gamma)-i-2}(\ell(\gamma)-i-2)!\\
\=&\mu(\gamma,\mathbf{1})+\sum_{i=0}^{\ell(\gamma)-2} (\ell(\gamma)-2)!(-1)^{\ell(\gamma)-2}\\
\=&\mu(\gamma,\mathbf{1})-\mu(\gamma,\mathbf{1})=0.
\end{align}
Correspondingly, the coefficient of~$\bigblok_{C\in \gamma}\vec{f_C}$  on the left-hand side of~(\ref{eq:recursion}) vanishes if there are~$C_1,C_2\in \gamma$ with~$1\in C_1$ and~$2\in C_2\,$. 
\end{proof}

\paragraph{Quasimodularity of pointwise products of moment functions}
Not only do the moment functions~$S_k$ admit quasimodular~$q$-brackets, but also the homogeneous polynomials in the moment functions admit quasimodular~$q$-brackets; here, each moment function~$S_k$ has weight~$k$ in accordance with the fact that~$\langle S_k\rangle_q$ has weight~$k$. Given a tuple~$\vec{k}=(k_1,...,k_n)$ of even integers, we write~$S_{\vec{k}} = S_{k_1}\cdots S_{k_n}.$ Note that, as a vector space,~$\SA$ is spanned by these functions~$S_{\vec{k}}$. We provide two approaches to proving the quasimodularity of the~$q$-brackets of the~$S_{\vec{k}}$. First, we give a direct proof of the statement in \cref{Sisquasimodular}, after which, in accordance with our main principle of establishing all identities before taking the~$q$-bracket, we prove a more general result which will be used frequently in the next section. 
\begin{thm}\label{Sisquasimodular}
The algebra~$\SA$ is a quasimodular algebra. More precisely, for $\vec{k}\in (2\n)^n$ one has
\begin{align}
\langle S_{\vec{k}}\rangle_q &\= \sum_{\alpha \in \Pi(n)}\prod_{A\in \alpha} D^{\ell(A)-1}G_{|\vec{k}_A|-2\ell(A)+2}\, . \label{eq:conSk2}\end{align}
\end{thm}
\begin{proof}
Observe that it suffices to show that
\begin{equation}\label{eq:conSk}\left\langle \bigotimes\nolimits_{k\in \vec{k}} S_{k}\right\rangle_q \= D^{n-1} G_{|\vec{k}|-2n+2}  \end{equation}
as~(\ref{eq:conSk2}) follows from~(\ref{eq:conSk}) by M\"obius inversion. Recall that~$\langle f_1\otimes \cdots \otimes f_n\rangle_q$ is the coefficient of~${x_1\cdots x_n}$ in~$\log\langle \exp(\sum_{i=1}^n x_i f_i)\rangle_q$ (see \cref{def:qbrac}(\ref{eq:conbraciii})). Consider~$S_k^0(\lambda)=\sum_{i=1}^\infty \lambda_i^{k-1}$ for all positive even~$k$. Euler's formula for the generating series of partitions
\[\sum_{\lambda\in \partitions} q^{|\lambda|} \= \prod_{m=1}^\infty (1-q^m)^{-1}\]
follows from writing $|\lambda|=\sum_{m\geq 1} m r_m(\lambda)$ and summing over all possible values of $r_1(\lambda), r_2(\lambda)$, etc. By the same idea, we find
\begin{align}\label{eq:expsk}\sum_{\lambda\in \partitions} \exp\left(\sum\nolimits_{k} S^0_{k}(\lambda)\, x_{k}\right)q^{|\lambda|} \= \prod_{m=1}^\infty \left(1-\exp\left(\sum\nolimits_{k} m^{k-1}x_{k}\right)q^m\right)^{-1}.\end{align}
The logarithm of this expression equals
\begin{equation}\label{eq:logof}\sum_{m,r=1}^\infty \exp\left(r\sum\nolimits_{k} m^{k-1}x_{k}\right)\frac{q^{mr}}{r}.\end{equation}
Now, assume all parts of~$\vec{k}$ are distinct. In the expansion of~(\ref{eq:logof}) the coefficient of~$x_{k_1}\cdots x_{k_n}$ equals
\[\sum_{m,r=1}^\infty  m^{|\vec{k}|-n} r^{n-1} q^{mr} \= D^{n-1}G_{|\vec{k}|-2n+2}\, .\]
Hence,
\[\left\langle \bigotimes\nolimits_{k\in \vec{k}} S_{k}^0\right\rangle_q \= D^{n-1} G_{|\vec{k}|-2n+2}.\]
By introducing distinct variables in~\cref{eq:expsk} for each repeated part of~$\vec{k}$ we obtain the same result if not all parts of~$\vec{k}$ are distinct.  

Note that if~$n\geq 2$, by \cref{lem:constantconnected} both sides of the equation do not change if one replaces~$S_k^0$ by~$S_k$. In case~$n=1$ we have established~(\ref{eq:conSk}) in \cref{prop:faulhaber} or in \cref{prop:ap2}. Hence,~(\ref{eq:conSk}) holds and~(\ref{eq:conSk2}) is then implied by M\"obius inversion.
\end{proof}

 Denoting 
\[ p_k(z) \= \begin{cases} \frac{z^{k-2}}{(k-2)!} & k\geq 0 \\ \frac{z^{-2}}{2} & k=0\end{cases}\]
and setting~$S_0(\lambda) \equiv 1$, one has the following expression for the generating series of the~$q$-bracket of the generators of~$\mathcal{S}$:
\begin{cor}
\[\sum_{k_1,\ldots,k_n\geq 0} \langle S_{k_1}\cdots S_{k_n}\rangle_q \, p_{k_1}(z_1)\cdots p_{k_n}(z_n) \= \sum_{\alpha\in \Pi(n)}\prod_{A\in \alpha} D^{|A|-1}\frac{P^{\text{even}}(\tau;z_A)}{2},\]
where~$z_A=\sum_{a\in A}z_a$ and 
\[P^{\text{even}}(\tau;z_1,\ldots, z_n) \= \frac{1}{2^n}\sum_{s\in \{-1,1\}^n} P(\tau,s_1z_1+\ldots+s_nz_n)\]
is the totally even part of the propagator in \upshape(\ref{eq:propagator}).
\end{cor}

\paragraph{Intermezzo: surjectivity of the~\texorpdfstring{$q$}{q}-bracket}
We deduce from \cref{Sisquasimodular} the surjectivity of the~$q$-bracket: every quasimodular form is the~$q$-bracket of some~$f\in \SA$. 

\begin{thm}\label{thm:surj}
The~$q$-bracket~$\langle\, \cdot\,\rangle_q:\SA\to \quasimodular$ is surjective. 
\end{thm}

Note that this is not obvious since the~$q$-bracket is not an algebra homomorphism. Denote by~$\vartheta_k:M_k\to M_{k+2}$ the Serre derivative, given by~$\vartheta_k=D+2kG_2$. Extend this notation by letting~$\vartheta_x:\quasimodular\to\quasimodular$ for~$x\in \q$ be given by~$\vartheta_x=D+2xG_2$. 
\begin{prop}\label{prop:20} Let~$x\in \q\backslash 2\z_{\geq 0}$. Then,
\[\quasimodular_k^{(\leq p)} \= \bigoplus_{r=0}^p \vartheta_x^r M_{k-2r}\, .\]
\end{prop}
\begin{proof}
Let~$f\in \modular_k$ with~$f\neq 0$. 
%By the quasimodular transformation property~(\ref{eq:modtrans})
%\[((Df)|_{k+2}\gamma)(z) = (Df)(z) + \frac{ck f(z)}{cz+d} \]
%for all~$z\in \mathfrak{H}$ and all~$\gamma=\left(\begin{smallmatrix} a &b \\ c & d\end{smallmatrix}\right)\in \Gamma$. Hence,~$(D+2x\phi)f$ 
Observe that~$\vartheta_xf$ is modular precisely if~$k=x$. By our assumption on~$x$, this is not the case. Hence,~$\vartheta_x$ increases the depth strictly by one. The result follows by induction on~$p$ by the same argument as in \cite[Proposition 20]{Zag08}. Namely, if~$\phi\in \quasimodular_k^{\leq p}$, then the last coefficient~$\phi_p$ in the quasimodular transformation~(\ref{eq:modtrans}) is a modular form of weight~$k-2p$. Hence,~$\phi$ is a linear combination of~$\vartheta_x^p \phi_p$ and a quasimodular form of depth strictly smaller than~$p$. 
\end{proof}

\begin{proof}[Proof of \cref{thm:surj}] First observe that~$(D+\Eis_2)\langle f\rangle_q = \langle S_2f\rangle_q\,$. As~$D+\Eis_2$ is not a Serre derivative, by Proposition~\ref{prop:20} it follows that it suffices to show that the~$q$-bracket is surjective on modular forms. Every modular form can be written as a polynomial of degree at most~$2$ in Eisenstein series, see \cite[Section 5]{Zag77}.
Hence, we show that the~$q$-bracket is surjective on polynomials of degree at most~$2$ in all Eisenstein series, possibly involving the quasimodular Eisenstein series~$\Eis_2$. 

Eisenstein series are in the image of the~$q$-bracket by \Cref{Sisquasimodular}. Note that~$D\Eis_{k}$ can be written a polynomial of degree~$2$ in Eisenstein series, explicitly:
\[D\Eis_k \= \frac{k+3}{2(k+1)}G_{k+2}-\sum_{\substack{0<j<k \\ j\equiv 1\, (2)}}\binom{k}{j} G_{j+1}G_{k+1-j}\, .\]
Also, we have an explicit formula for the~$q$-bracket of~$S_kS_l$:
\begin{equation}\label{eq:SkSl}\langle S_{k}S_{l}\rangle_q \= \Eis_k\Eis_l+D\Eis_{k+l-2}\, ,\end{equation}
 so that this~$q$-bracket is expressible as a polynomial of degree at most~$2$ in the Eisenstein series. 
 
 Now fix an integer~$m\geq 4$. We consider the equations~(\ref{eq:SkSl}) for all~$k+l=m$. It suffices to show that we can invert these equations, i.e., write~$\Eis_k\Eis_l$ as a linear combination of~$q$-brackets of products of at most two~$S_i$. A direct computation shows that the determinant of the matrix corresponding to the equations above equals 
\[1-\sum_{\substack{0<j<m \\ j\equiv 1\, (2)}}\binom{m}{j}\=1-2^{m-3}\;<\;0.\] 
Hence, the~$q$-bracket is surjective.
\end{proof}
\begin{remark}
Only the last step of above proof uses the explicit formula~(\ref{eq:SkSl}) for the derivative of Eisenstein series. The author expects one could conclude the proof by an abstract argument, but he is not aware of such an argument. 
\end{remark}

\paragraph{The connected product of moment functions}

In the second approach we compute the connected product~$S_{k_1}\cp \ldots\cp S_{k_n}$, which by \cref{prop:prod} yields the left-hand side of~(\ref{eq:conSk2}) after taking the~$q$-bracket. The result is formulated in \cref{prop:computephi} below and depends on two technical lemma's which we state first. 

In order to do so, we start by introducing the following notation. For a partition~$\lambda$ and a subset~$A$ of~$\n$, we write~$\lambda|_A$ for the partition where a part of size~$m$ occurs~$r_m(\lambda)$ times if~$m\in A$ and does not occur if~$m\not \in A$. For example,~$(5,4,3,3,1,1,1)|_{\{4,1\}}=(4,1,1,1).$ 
\begin{defn}
We say~$f:\partitions\to \q$ is \emph{supported on}~$A$ if~$f(\lambda)=f(\lambda|_A)$ for all partitions~$\lambda$. 
\end{defn} 

The first lemma expresses the induced product of two functions~$F$ and~$G$ supported on disjoint sets as the \emph{pointwise product} of these functions, and of two functions~$F$ and~$G$ supported on the same singleton set as a \emph{convolution product} of functions. 

\begin{lem}\label{lem:disjointsupport} Suppose~$X$ and~$Y$ are subsets of~$\n$ and~$F,F',G,G':\partitions\to\q$ are supported on~$X, X, Y$ and~$Y$, respectively. 
Then
\begin{enumerate}[\upshape (i)]
\item $F\blok F'$ is supported on~$X$;
\item \label{case:i} If~$X$ and~$Y$ are disjoint, then 
\[FG\blok F'G'\=(F\blok F')(G\blok G'), \quad\quad \text{in particular} \quad \quad F\blok G \=FG;\] 
\item \label{case:ii} If~$X=Y=\{m\}$, then
	\[(F\blok G)(\lambda)\=\partial (f\conv g)(r_m(\lambda)),\]
where~$f$ and~$g$ are such that~$F(\lambda)=f(r_m(\lambda))$,~$G(\lambda)=g(r_m(\lambda))$. 
\end{enumerate}

\end{lem}
\begin{proof} By \cref{prop:defblok2}, we have
\[(F\blok F')(\lambda) \= \sum_{\alpha\cup \beta\cup\gamma=\lambda}(-1)^{\ell(\gamma)}\,F(\alpha)\,F'(\beta),\]
where it is understood that~$\gamma$ is a strict partition. We have that
\begin{align}(F\blok F')(\lambda) &\= \Bigl(\,\sum_{\alpha\cup \beta\cup\gamma=\lambda|_{X}}(-1)^{\ell(\gamma)}\,F(\alpha)\,F'(\beta)\Bigr)\Bigl(\,\sum_{\alpha\cup \beta\cup\gamma=\lambda|_{X^c}}(-1)^{\ell(\gamma)}\Bigr) \\
&\= (F\blok F')(\lambda|_{X})\cdot(1\blok 1)(\lambda|_{X^c}).\end{align}
Recall~$f\blok 1=f$ for all functions~$f$, hence~$(F\blok F')(\lambda) = (F\blok F')(\lambda|_{X})$, which is the first statement. 

Next, we have that
\[(FG\blok F'G')(\lambda) \= \sum_{\alpha\cup \beta\cup\gamma=\lambda} (-1)^{\ell(\gamma)}\, (FG)(\alpha)\, (F'G')(\beta),\]
where again it is understood that~$\gamma$ is a strict partition. Using the fact that~$F,F',G$ and~$G'$ are supported on~$X,X,Y$ and~$Y$, respectively, we obtain
\begin{equation}\label{eq:FGblokF'G'} (FG\blok F'G')(\lambda) \= \sum_{\alpha\cup \beta\cup\gamma=\lambda} (-1)^{\ell(\gamma|_X)+\ell(\gamma|_Y)+\ell(\gamma|_{Z})}\,F(\alpha|_X)\,G(\alpha|_Y)\,F'(\beta|_X)\,G'(\beta|_Y),\end{equation}
where~$Z$ denotes the complement of~$X\cup Y$ in~$\n$. We factor the right-hand side of~(\ref{eq:FGblokF'G'}) as
\[\Bigl(\,\sum_{\alpha\cup \beta\cup\gamma=\lambda|_X} (-1)^{\ell(\gamma)}\,F(\alpha)\,F'(\beta)\Bigr)\Bigl(\,\sum_{\alpha\cup \beta\cup\gamma=\lambda|_Y} (-1)^{\ell(\gamma)}\,G(\alpha)\,G'(\beta)\Bigr)\Bigl(\sum_{\alpha\cup \beta\cup\gamma=\lambda|_{Z}} (-1)^{\ell(\gamma)}\Bigr).\]
By definition of the product~$\blok$, we conclude 
\[(FG\blok F'G')(\lambda) \= (F\blok F')(\lambda|_X)\,(G\blok G')(\lambda|_Y)\,(1\blok 1)(\lambda|_{Z})\=(F\blok F')(\lambda)\,(G\blok G')(\lambda).\] 

By taking~$F'$ and~$G$ to be the constant function~1 (which is supported on every~$X$ and~$Y$), we see that~$F\blok G' =FG'$ is implied by~$FG\blok F'G'=(F\blok F')(G\blok G')$.

Next, for~\ref{case:ii} we have
\begin{align}
(F\blok G)(\lambda) &= \sum_{\alpha\cup\beta\cup \gamma=\lambda}  (-1)^{\ell(\gamma)} f(r_m(\alpha))\,g(r_m(\beta))\\
&\= \sum_{\alpha\cup\beta\cup \gamma=\lambda|_{\{m\}}} (-1)^{\ell(\gamma)} f(r_m(\alpha))\,g(r_m(\beta)) \end{align}
Letting~$i=r_m(\alpha)$ and~$j=r_m(\beta)$, we have
\begin{align} (F\blok G)(\lambda)&= \sum_{i+j=r_m(\lambda)} f(i)\,g(j)-\sum_{i+j+1=r_m(\lambda)}f(i)\,g(j)\\
&\= (f\conv g)(r_m(\lambda))-(f\conv g)(r_m(\lambda)-1)\\
&\= \partial (f\conv g)(r_m(\lambda)). \qedhere
\end{align}
\end{proof}

The second lemma is concerned with the vanishing of certain sums of the M\"obius functions of \emph{set} partitions. Given~$\alpha\in \Pi(n)$ and a subset~$Z$ of~$[n]$, we let 
\[\alpha|_Z=\{A\cap Z \mid A\in \alpha \text{ s.t.~} A\cap Z\neq \emptyset\} \in \Pi(Z),\]
 where~$\Pi(Z)$ denotes the set of all partitions of the set~$Z$. Observe that \[\ell(\alpha)\=\ell(\alpha|_Z)+|\{A\in \alpha\mid A\cap Z=\emptyset\}|,\] in particular~$\ell(\alpha|_Z)\leq \ell(\alpha)$. 
Given~$Z\subset[n]$, define an equivalence relation on~$\Pi(n)$ by writing~$\alpha\sim \beta$ if
\begin{equation}\label{eq:equivalence} \alpha|_Z=\beta|_Z \quad \text{and} \quad \alpha|_{Z^c} = \beta|_{Z^c}.\end{equation}
\begin{lem}\label{lem:hypergeometric} Let $Z\subseteq [n]$. If~$Z\neq \emptyset$ and $Z\neq [n]$, then for all~$\beta\in \Pi(n)$ we have
\[\sum_{\alpha\sim\beta}\mob=0.\]
\end{lem}
\begin{proof}
 %This will imply that we can restrict the first sum in~(\ref{eq:27}) to~$\vec{m}\in \n^n$ for which~$m_i=m_j$ for all~$i,j$. 
Observe that~$\alpha\sim\beta$ precisely if for all~$A\in \alpha$ we have ($A\cap Z=\emptyset$ or~$A\cap Z\in\beta|_Z$) and similarly we have ($A\cap Z^c=\emptyset$ or~$A\cap Z^c\in\beta|_{Z^c}$). Hence, every~$A\in \alpha$ is the union of some~$A_1\in \alpha|_Z\cup \{\emptyset\}$ and~$A_2\in \alpha|_{Z^c}\cup \{\emptyset\}$ with not both~$A_1=\emptyset$ and~$A_2=\emptyset$. Write~$a=\ell(\beta|_{Z})$,~$b=\ell(\beta|_{Z^c})$, and assume without loss of generality that~$a\leq b$. Write~$k$ for the number of~$A\in \alpha$ for which both~$A_1\neq \emptyset$ and~$A_2\neq \emptyset$. Now,~$\ell(\alpha)=a+b-k$. Moreover, given~$k$,~$Z$ and~$\beta$, there are 
\[\binom{a}{k}\binom{b}{k}k!\]
ways to choose~$\alpha \sim \beta$ with~$\ell(\alpha)=a+b-k$. Hence, we find 
\begin{align}
\sum_{\alpha\sim \beta}\mob &\= \sum_{k=0}^a (-1)^{a+b-k-1}(a+b-k-1)!\, \binom{a}{k}\binom{b}{k}k!\\
&\=(-1)^{a+b-1}(a+b-1)!\sum_{k=0}^a \frac{(-a)_k(-b)_k}{(-a-b+1)_k(1)_k},
\end{align}
where~$(d)_k=\prod_{i=0}^{k-1} (d+i)$ is the rising Pochhammer symbol. This expression equals up to the constant~$(-1)^{a+b-1}(a+b-1)!$ the special value~${}_2 F_1(-a,-b,-a-b+1;1)$ of the hypergeometric function~${}_2 F_1(-a,-b,-a-b+1;z)$, which vanishes by Gauss's theorem subject to $a,b>0$. As $Z\neq \emptyset$, we have~$a>0$. Also, $b>0$ as $Z\neq[n]$.
\end{proof}

The following result not only computes the connected product of the moment functions~$S_k$, but also is one of the main technical results needed to prove \cref{thm:1}. 
\begin{thm} \label{prop:computephi} Let~$k_i,f_i$ for~$i=1,\ldots, n$ be such that {\upshape(\ref{def:Tkf})} defines~$S_{k_i,f_i}$. Then,
\begin{enumerate}[\upshape(i)]
\item\label{prop:computephi-1} There exists a function~$g:\n \to \q$ such that 
\[ S_{k_1,f_1}\cp \ldots\cp  S_{k_n,f_n}\=S_{|\vec{k}|,g}.\] 
In fact, 
\[g\=\sum_{\alpha\in\Pi(n)} \mob \ \partial^{\ell(\alpha)-1}\bigconv_{A\in \alpha} f_A,\]
 where~$f_A = \prod_{a\in A} f_{a}$ and~$\bigconv$ denotes the convolution product {\upshape(\ref{eq:conv})}.
\item\label{prop:computephi-i} If~$f_1(x)=x$, then~$\partial g = f_1\, \partial \tilde{g}$ with~$\tilde{g}$ given by~$S_{k_2,f_2}\cp \ldots\cp  S_{k_n,f_n}=S_{|\vec{k}|,\tilde{g}}.$
\end{enumerate}
\end{thm}
\begin{remark}
We extend~$g$ by~$g(0)=0$. Here and later in this work, we usually omit the dependence of~$g$ on~$f_1,\ldots, f_n$ in the notation.
\end{remark}
\begin{proof}
For the first part, we let~$\vec{m}^{\vec{k_A}} \, \vec{f_A} \circ r_{\vec{m}}$  denote~$\prod_{i}m_i^{k_{A_i}} \cdot f_{A_i}\circ r_{m_i},$ where~$r_{m_i}$ is considered as a function~$\partitions \to \q$. In case~$n=1$ the result~(\ref{prop:computephi-1}) is trivially true, so we assume $n\geq 2$. By definition of the connected product and~$S_{k,f}$ (see~(\ref{def:connectedprod}) and~(\ref{def:Tkf}) respectively) we have
\begin{align}\
S_{k_1,f_1}\cp \ldots\cp  S_{k_n,f_n} &\=\sum_{\alpha\in\Pi(n)} \mob \bigblok_{A\in\alpha} \,\Bigl(\, \sum_{\vec{m}\in \n^{\ell(A)}} \vec{m}^{\vec{k_A}} \, \vec{f_A} \circ r_{\vec{m}}\Bigr)\\ 
\label{eq:conprodcomp} &\=\sum_{\vec{m}\in\n^n}\sum_{\alpha\in\Pi(n)} \mob \bigblok_{A\in\alpha} \vec{m_A}^{\vec{k_A}}\, \vec{f_A} \circ r_{\vec{m}}\, .
\end{align}
For all $m\geq 0$, the function $r_{m}:\partitions\to\q$ is supported on $\{m\}$. Having \cref{lem:disjointsupport} in mind, we aim to factor the functions in~\eqref{eq:conprodcomp} as a product of functions supported on a singleton set. Given~$\vec{m}\in \n^n$, we start by all functions supported on $\{m_1\}$, that is, we let~$Z(\vec{m})=\{i \mid m_i=m_1\}\subset[n]$. Note that~$Z(\vec{m})$ determines all~$i$ for which the support of $r_{m_i}$ contains $m_1$. Denote by~$E(\vec{m})$ the set of equivalence classes of~$\Pi(n)$ for this choice of~$Z=Z(\vec{m})$. We split the sum over~$\alpha\in \Pi(n)$ in~(\ref{eq:conprodcomp}) as a sum over the elements of~$E(\vec{m})$, i.e.
\begin{align}\label{eq:disjointsupport}
S_{k_1,f_1}\cp \ldots\cp  S_{k_n,f_n} &\=\sum_{\vec{m}\in\n^n}\sum_{[\beta]\in E(\vec{m})}\sum_{\alpha\in [\beta]} \mob \bigblok_{A\in\alpha} \vec{m_A}^{\vec{k_A}} \, \vec{f_A}\circ r_{\vec{m_A}}.
\end{align}
Then, given~$\vec{m}\in\n^n$,~$Z=Z(\vec{m})$ and~$A\in \alpha|_Z$, the function~$\lambda \mapsto \vec{m_A}^{\vec{k_A}}\, \vec{f_A}(r_{\vec{m_A}}(\lambda))$ is supported on~$\{m_1\}$, whereas for~$A\in \alpha|_{Z^c}$ the function~$\lambda \mapsto \vec{m_A}^{\vec{k_A}} \vec{f_A}(r_{\vec{m_A}(\lambda)})$ is supported on~$\n\backslash\{m_1\}$. Hence, by \cref{lem:disjointsupport}(\ref{case:i}) we find that~(\ref{eq:disjointsupport}) equals
\begin{align}\label{eq:27}
\sum_{\vec{m}\in\n^n}\sum_{[\beta]\in E(\vec{m})}\sum_{\alpha\in [\beta]} \mob 
\biggl(\, \bigblok_{A\in\alpha|_Z} \vec{m_A}^{\vec{k_A}} \, \vec{f_A} \circ r_{\vec{m_A}}\biggr)
\biggl(\,\bigblok_{A\in\alpha|_{Z^c}} \vec{m_A}^{\vec{k_A}} \, \vec{f_A}\circ r_{\vec{m_A}}\biggr). 
\end{align}
Instead of writing the second factor as a product of functions which are all supported on a singleton set, we make the following observation. 

As~$\alpha|_Z=\beta|_Z$ and~$\alpha|_{Z^{c}}=\beta|_{Z^c}$, the only dependence on~$\alpha$ in the above equation is in~$\mob$.  
By construction~$Z(\vec{m})$ is non-empty. Hence, by \cref{lem:hypergeometric} we have that if~$Z\neq [n]$ then for all~$\beta\in E(\vec{m})$ we have~$\sum_{\alpha\in [\beta]}\mob=0.$ This implies that we can restrict the first sum in~(\ref{eq:27}) to~$\vec{m}\in \n^n$ for which~$m_i=m_j$ for all~$i,j$, that is
 %Hence, we have shown that 
\begin{align}
S_{k_1,f_1}\cp \ldots\cp  S_{k_n,f_n} &\=
\sum_{m\in\n}\sum_{\alpha\in \Pi(n)} \mob \bigblok_{A\in\alpha} \prod_{a\in A}m^{k_a}\cdot{f_{a}}\circ r_{m}\, .
\end{align}
Applying  \cref{lem:disjointsupport}\ref{case:ii}~$\ell(\alpha)-1$ times and using~(\ref{convleibniz}), we obtain the desired result.

For the second part, let~$Z=\{1\}$ and consider an equivalence class~$[\beta]$ for the equivalence relation~\eqref{eq:equivalence} determined by~$Z$. We split the sum 
\[\partial g \= \sum_{\alpha\in\Pi(n)} \mob\, \partial^{\ell(\alpha)}\bigconv_{A\in \alpha} f_A\]
over all conjugacy classes. Write~$A_1$ for the element of~$\alpha$ for which~$1\in A_1$. Denote~$\hat A_1=A_1\backslash\{1\}$ and~$\gamma=\beta|_{\{2,\ldots,n\}}$. 
In case~$A_1=\{1\}$ one has by~(\ref{eq:f*id}) that
\begin{equation}\label{eq:dgi}\mob\ \partial^{\ell(\alpha)}\bigconv_{A\in \alpha} f_A \= -\ell(\gamma)\,\mu(\gamma,\mathbf{1})\ \partial^{\ell(\gamma)-1}(1-\partial)\bigconv_{A\in \gamma}f_A\, .\end{equation}
In case~$A_1\neq \{1\}$ (i.e.~$|A_1|\geq 2$), one finds by~(\ref{prodleibniz}) that
\begin{equation}\label{eq:dgii}\mob\ \partial^{\ell(\alpha)}\bigconv_{A\in \alpha} f_A \=\mu(\gamma,\mathbf{1})\ \partial^{\ell(\gamma)-1}(f_1\, \partial f_{\hat A_1}+(1-\partial)f_{\hat A_1})\conv\!\bigconv_{A\in \gamma\backslash\hat A_1}f_A\, .\end{equation}
As~$[\beta]$ contains one element for which~(\ref{eq:dgi}) holds and~$\ell(\gamma)$ elements for which~(\ref{eq:dgii}) holds, one finds		
\[\sum_{\alpha\in [\beta]}\mob \partial^{\ell(\alpha)}\bigconv_{A\in \alpha} f_A \= 
\mu(\gamma,\mathbf{1})\ \partial^{\ell(\gamma)-1}\sum_{C\in \gamma}\Bigl(f_1\, \partial f_{C}\conv\!\bigconv_{A\in \gamma\backslash C}f_A\Bigr). \]
By~(\ref{convleibniz}) and~(\ref{eq:id(f*g)}) this equals 
\[ \mu(\gamma,\mathbf{1})\sum_{C\in \gamma}\Bigl(f_1 \, \partial f_{C}\conv\!\bigconv_{A\in \gamma\backslash C}\partial f_A\Bigr)\= \mu(\gamma,\mathbf{1})\,f_1\ \partial^{\ell(\gamma)}\bigconv_{A\in \gamma}f_A\, . \]
Hence, summing over all conjugacy classes, we obtain
\begin{align}\partial g &\= f_1\sum_{\gamma \in \Pi(n-1)}  \mu(\gamma,\mathbf{1})\ \partial^{\ell(\gamma)}\bigconv_{A\in \gamma}f_A \=f_1\ \partial \tilde g. \qedhere\end{align}
\end{proof}

The case when~$f_1(x)=\ldots=f_n(x)=x$ is the easiest example (for arbitrary $n\in \n$) of the above result. In this case one generalises \cref{Sisquasimodular} by a result which, in accordance with our main principle of establishing identities before the~$q$-bracket, yields this theorem after taking the~$q$-bracket.
\begin{cor}\label{lem:computephi} For all positive even~$k_1,\ldots, k_n$ one has
\[S_{k_1}\cp \ldots\cp  S_{k_n}\=S_{|\vec{k}|-n,\Faulhaber_n}\, .\]
\end{cor}
\begin{proof}
Recall~$S_k=S_{k-1,\mathrm{id}}$ and apply \cref{prop:computephi}(\ref{prop:computephi-i})~${n-1}$ times. 
\end{proof}

Later we will use \cref{Sisquasimodular} when the~$f_i$ are Faulhaber polynomials. This is the situation in which we prove the main result of this paper, in which case the following lemma is useful. 

\begin{lem}\label{lem:convispol}
%\sloppy 
If~$f_1,\ldots,f_n$ are Faulhaber polynomials of degrees~$d_1,\ldots, d_n$, respectively and ${g:\n\to\q}$ is as in \cref{prop:computephi}, then there exists a polynomial~$p$ such that~$\partial g(m)=p(m)$ for all~$m\in \n$. Moreover,~$p$ is strictly of degree~$|\vec{d}|-1$, is even or odd and~$p(0)=0$. 
\end{lem}
\begin{proof}
By \cref{prop:computephi}(\ref{prop:computephi-i}) we can assume w.l.o.g.\ that none of the degrees~$d_i$ equals~$1$. Now, consider a monomial~$\partial^{\ell(\alpha)}\bigconv_{A\in \alpha} f_A$ in~$\partial g$. Note that both~$\conv$ and~$\partial$ are operators on the space of polynomials, more precisely:
\[
\conv:\q[x]_{\leq k}\times \q[x]_{\leq l} \to \q[x]_{\leq k+l+1} \quad\quad \text{and}\quad\quad\partial:\q[x]_{\leq k} \to \q[x]_{\leq k-1}
\]
as
\[x^k*x^l\=\frac{k!l!}{(k+l+1)!}x^{k+l+1}+O(x^{k+l})\quad\quad \text{and}\quad\quad\partial(x^k)\=kx^{k-1}+O(x^{k-2}).\]
Hence, the degree of such a monomial is~$|\vec{d}|-1$. Now observe that by the symmetry~(\ref{eq:faulhabersym}) one has
\[ \partial f_A(x) \= f_A(x)-f_A(x-1) \= f_A(x)-(-1)^{|A|}f_A(-x).\]
Therefore, we see that~$\partial f_A$ is even or odd and as the convolution product preserves this property, every monomial is even or odd. By the same arguments~$\partial f_A(0)=0$ and hence the constant term of every monomial vanishes. Therefore, every monomial~$\partial^{\ell(\alpha)-1}\bigconv_{A\in \alpha} f_A$ in~$g$ satisfies the desired properties, so that it remains to show that the leading coefficient does not vanish. 

As~$\Faulhaber_l = \frac{1}{l}x^l+O(x^{l-1})$, the leading coefficient of a monomial as above equals
\begin{align}&\frac{|\vec{d}|}{\prod_{i}d_i}\frac{\prod_{i=1}^n d_{A_i}!}{|\vec{d}|!},\end{align}
where for a set~$B$ we have set~$d_{B}=\sum_{b\in B}d_b$.
Hence, the leading coefficient of~$\partial g$ equals
\begin{equation}\label{eq:col}\frac{|\vec{d}|}{\prod_{i}d_i}\cdot \sum_{\alpha\in\Pi(n)} \mob \binom{|\vec{d}|}{d_{A_1},\ldots,d_{A_r}}^{-1},\end{equation}
where~$\alpha=\{A_1,\ldots,A_r\}$. Note that this number has the following combinatorial interpretation. Let~$n$ balls be given which are colored such that~$d_1$ balls are colored in the first color,~$d_2$ in the second color, etc. Suppose we use the same multiset of colors to additionally mark each ball with a dot (possibly of the same color), that is,~$d_1$ balls are marked with a dot of the first color,~$d_2$ with a dot of the second color, etc.  Given a subset~$C$ of the set of all colors, it may happen that if we consider all balls colored by the colors of~$C$, all the dots on these balls are colored by the same set of colors~$C$. We then say that the balls are \emph{well-colored with respect to~$C$}. For example, both the empty set of colors and the set of all possible colors give rise to a well-coloring of balls.
If we independently at random color and mark the balls as above, the probability that the balls colored by a subset~$C$ are well-colored is~$\binom{|\vec{d}|}{d_C}^{-1}$. Hence, by applying M\"obius inversion the number \[\sum_{\alpha\in\Pi(n)} \mob \binom{|\vec{d}|}{d_{A_1},\ldots,d_{A_r}}^{-1}\]
 equals the probability that if we independently at random color and mark the balls as above, there does not exist a proper non-empty subset~$C$ of the colors such that the balls colored by~$C$ are well-colored. If we mark at least one ball of every color~$i$ with color~$i+1$ (modulo~$n$), such a set~$C$ cannot exist. Hence, the number~(\ref{eq:col}) is positive, so the polynomial~$p$ is strictly of degree~$|\vec{d}|-1$. 
\end{proof}

\section{Three quasimodular algebras}\label{sec:3}
\paragraph{Introduction}
Given integers~$k,l$ with~$k\geq 0$ and~$l\geq 1$ recall the definition of the \emph{double moment functions} in~(\ref{def:Tkl0}) 
by
\[T_{k,l}(\lambda) \= -\frac{B_{k+l}}{2(k+l)}(\delta_{l,1}+\delta_{k,0}) + \sum_{m=1}^\infty m^k \Faulhaber_l(r_m(\lambda)).~\]
Unless stated explicitly, we always assume that
\begin{equation}\label{eq:cond} k\in \z_{\geq 0}, l\in \z_{\geq 1}, k+l\in 2\z. \end{equation} 
Moreover, it turns out to be useful to define~$T_{0,0}\equiv T_{-1,1} \equiv -1$ and~$T_{k,l}\equiv 0$ for other pairs~$(k,l)$ with~$k<0$ or~$l<1$. 
 
\begin{remark} The double moment functions specialize to the moment functions studied in the previous section whenever~$l=1$, i.e.~$T_{k,1}=S_{k+1}$. Also, as~$\Faulhaber_l(1)=1$, for a strict partition~$\lambda$ %(i.e.~$r_m(\lambda)\in \{0,1\}$ for all~$m$) 
one has~$T_{k,l}(\lambda)=S_k(\lambda)$. Hence, our functions~$T_{k,l}$ can be seen as an extension of the \emph{algebra of supersymmetric polynomials}, mentioned in the introduction, to functions on all partitions (and not only on strict partitions). 
\end{remark}
\begin{remark}
In case $k+l$ is odd, the~$q$-bracket of~$T_{k,l}$ does not vanish---in contrast to the shifted symmetric functions for which the~$q$-bracket vanishes for all odd weights. However, the~$q$-bracket of a polynomial involving the double moment functions in both even and odd weights also is a polynomial in the so-called combinatorial Eisenstein series, defined in \cref{defn:combeis}.
\end{remark}

These double moment functions give rise to three different graded algebras, which turn out to be quasimodular (see page~\pageref{defn:quasimodularalgebra}). 
\begin{defn}\label{defn:3alg}
Define the~$\q$-algebras~$\SA, \Sym^\blok(\SA)$ and~$\TA$ by the condition that
\begin{itemize}\itemsep0pt
\item $\SA$ is generated by the moment functions~$S_k$ under the pointwise product;
\item $\Sym^\blok(\SA)$ is generated by the elements of~$\SA$ under the induced product;
\item $\TA$ is generated by the double moment functions under the pointwise product.
\end{itemize}
\end{defn}
Our main result \cref{thm:1} is slightly refined by the following statement. 
\begin{thm}\label{thm:mainprop}
Let~$X$ be any of the algebras~$\SA,\Sym^\blok(\SA)$ and~$\TA$. Then,~$X$ is
\begin{itemize}\itemsep0pt
\item quasimodular;
\item closed under the pointwise product;
\item closed under the induced product if~$X\neq \SA$.
\end{itemize}
Moreover, the three algebras are related by~$\displaystyle \SA\subsetneq \Sym^\blok(\SA) \subsetneq \TA$.
\end{thm}

\begin{remark}
Observe that being closed under the pointwise product is not implied by being quasimodular. For example, the algebra~$\mathcal{R}=\langle\, \cdot\,\rangle_q^{-1}(\widetilde{M})$ in \cref{rk:r} is quasimodular, closed under the induced product and~$\TA\subset \mathcal{R}$, but~$\mathcal{R}$ is not closed under the pointwise product \cite[Section 9]{Zag16}. 
\end{remark}

In the next section we provide different bases for these algebras: in this way we obtain many examples of functions with a quasimodular~$q$-bracket, and moreover, the study of these bases leads to a proof of \cref{thm:mainprop}.

\begin{remark}\label{rk:comparison}
The algebras~$\TA$ and~$\Lambda^*$ are different algebras, as follows from the observation that~$f(\lambda)=(-1)^k f(\lambda')$ for all~$f\in \Lambda_k^*$, which follows by writing a shifted symmetric polynomial as a symmetric polynomial in the Frobenius coordinates. This does not hold for all~$f\in \TA$, as can easily be checked numerically. On the other hand, it is not true that~$f(\lambda)\neq \pm f(\lambda')$ for all~$f\in \TA$, as~$Q_2=T_{1,1}$ with~$Q_k$ defined by Equation~(\ref{eq:qk}). 
More precisely, one has
\[ \TA\cap \Lambda^* = \q[Q_2].\]
Namely, if~$f\in\TA\cap \Lambda^*$, consider a \emph{strict} partition~$\lambda$ (i.e., a partition for which $r_m(\lambda)\leq 1$ for all~$m$). Then, we have that~$f(\lambda)$ is symmetric polynomial in the parts $\lambda_1,\lambda_2,\ldots$. On the other hand, as~$f\in \Lambda^*$, it follows that~$f(\lambda)$ is a shifted symmetric polynomial in the parts~$\lambda_1,\lambda_2,\ldots$. The only polynomials of degree~$d$ in the variables~$x_i$ that are both symmetric and shifted symmetric are up to a constant given by
$\left(\sum_{i} x_i\right)^d$, hence~$f\in \q[Q_2]$.
\end{remark}

\paragraph{The basis given by double moment functions}
 In this section we show that~$\TA$ is closed under the induced product. Moreover, we show that~$\SA$ and~$\Sym^\blok(\SA)$ are subalgebras of~$\TA$. In the next section, we use these results to define a weight grading on~$\TA$. Observe that as a vector space~$\mathcal{T}$ is spanned by the functions~$T_{\vec{k},\vec{l}}$, defined by~$T_{\vec{k},\vec{l}}=\prod_{i} T_{k_i,l_i}$, for all~$\vec{k}, \vec{l}\in \z^n$ satisfying the conditions~(\ref{eq:cond}) for all pairs~$(k,l)=(k_i,l_i)$. 

\begin{thm}\label{thm:tclosed}
The algebra~$\TA$ is closed under the induced product.
\end{thm}
\begin{proof}
Observe that
\[T_{\vec{k},\vec{l}}\blok T_{\vec{k}',\vec{l}'} \= T_{\vec{k},\vec{l}} T_{\vec{k}',\vec{l}'} -T_{\vec{k},\vec{l}}\cp  T_{\vec{k}',\vec{l}'}\, .\]
Hence, it suffices to show that~$T_{\vec{k},\vec{l}}\cp  T_{\vec{k}',\vec{l}'}$ can be expressed in terms of elements of~$\TA$.

By \cref{prop:computephi} and \cref{lem:convispol}, we have that an expression of the form:
$$T_{k_1,l_1}\cp \cdots\cp T_{k_n,l_n}$$
is an element of~$\TA$. \Cref{prop:connectedbracketrecursion} implies that~$f_1f_2\cp f_3\cp f_4\cp \ldots\cp  f_n$ equals
\begin{align} (f_1\cp f_2\cp \ldots\cp  f_n) \+ \sum_{A\sqcup B=\{3,\ldots, n\}}&\bigl( (f_1\cp f_{A_1}\cp f_{A_2}\cp \ldots)\cdot(f_2\cp f_{B_1}\cp f_{B_2}\cp \ldots)\\&-(f_1\cp f_{A_1}\cp f_{A_2}\cp \ldots)\,\big|\, (f_2\cp f_{B_1}\cp f_{B_2}\cp \ldots)\bigr).\end{align}
Hence, by using this proposition recursively, we can replace the pointwise products in~$T_{\vec{k},\vec{l}}$ and~$T_{\vec{k}',\vec{l}'}$ by a linear combination of connected products of double moment functions~$T_{k,l}$, showing that~$T_{\vec{k},\vec{l}}\cp  T_{\vec{k}',\vec{l}'}$ is an element of~$\TA$. 
\end{proof}

Now, we determine a basis for the three algebras. Let~$\TA^\mathrm{mon}$ be the set of all monomials for the pointwise product in~$\TA$. Two elements of~$\TA^\mathrm{mon}$ are considered to be the same if one can reorder the products so that they agree, for example~$T_{1,1}T_{3,5}$ and~$T_{3,5}T_{1,1}$ are the same function. In other words, every elements of~$\TA^\mathrm{mon}$ can be written as~$T_{\vec{k},\vec{l}}$ in a unique way up to commutativity of the (pointwise) product.

\begin{thm}
We have 
\begin{align}\label{eq:subalg}\SA\subsetneq \Sym^\blok(\SA) \subsetneq \TA.\end{align}
Moreover, a basis for 
\begin{itemize}\itemsep0pt
\item $\TA$ is given by~$\TA^\mathrm{mon}$;
\item $\Sym^\blok(\SA)$ is given by all~$T_{\vec{k},\vec{l}}\in \TA^\mathrm{mon}$ satisfying~$k_i\geq l_i$ for all~$i$;
\item $\SA$ is given by all~$T_{\vec{k},\vec{l}}\in \TA^\mathrm{mon}$ satisfying~$l_i=1$ for all~$i$.
\end{itemize}
%A basis for the space~$\TA$ is given by the~$T_{\vec{k},\vec{l}}$ with~$k_i+l_i$ even. The space~$\SA^{\blok}$ is the subspace given by the~$T_{\vec{k},\vec{l}}$ for which~$k_i\geq l_i$. \comm{Up to permuting the order in the product}
\end{thm}
\begin{proof} It suffices to prove the second part, as from the stated bases statement~(\ref{eq:subalg})  follows immediately. 

By definition the elements of~$\TA^\mathrm{mon}$ generate~$\TA$ as a vector space. Hence, it suffices to show that they are linearly independent, i.e., that if 
\begin{align}\label{eq:indep} \sum_{\alpha\in I} c_\alpha T_\alpha(\lambda) \=0\end{align}
for all~$\lambda\in \partitions$, where~$I$ is the set of all pairs~$(\vec{k},\vec{l})$ up to simultaneous reordering and~$c_\alpha\in \q$, we have that~$c_\alpha=0$ for all~$\alpha$. 

First of all, let~$\lambda=(N_1,N_2)$ and consider~(\ref{eq:indep}) as~$N_1\to \infty$. Note that~$T_{\vec{k},\vec{l}}(\lambda)$ grows as~\[N_1^{|\vec{k}|}+N_2^{k_{\mathrm{min}}}N_1^{|\vec{k}\backslash k_{\mathrm{min}}|}\] plus lower order terms, where~$k_{\mathrm{min}}$ is the smallest of the~$k_i$ in~$\vec{k}$. Hence,~$|\vec{k}|$ should be constant among all~$T_\alpha$ in~(\ref{eq:indep}). Moreover, %by considering the next term in the growth of~$T_{\vec{k},\vec{l}}$ 
we conclude that~$k_{\mathrm{min}}$ should be constant among all~$T_\alpha$ in~(\ref{eq:indep}). Continuing by considering the lower order terms, we conclude that~$\vec{k}$ is constant among all~$T_\alpha$. Similarly, by instead considering partitions consisting of~$N_1$ times the part~$1$ and~$N_2$ times the part~$2$, we conclude that~$\vec{l}$ is constant among all~$T_{\alpha}$. Hence, there is at most one~$\alpha$ with nonzero coefficient~$c_\alpha$. We conclude that~$c_\alpha=0$ for all~$\alpha \in I$. 

For~$\Sym^\blok(\SA)$ we show, first of all, that indeed~$T_{\vec{k},\vec{l}}\in \Sym^\blok(\SA)$ if~$k_i\geq l_i$ for all~$i$. Let~$k\geq l$ of the same parity be given. By \cref{lem:computephi} we find that
\[\underbrace{T_{1,1}\cp T_{1,1}\cp \ldots\cp  T_{1,1}}_{l-1}\cp T_{k-l+1,1}\=\underbrace{S_2\cp S_2\cp \ldots\cp  S_2}_{l-1}\cp S_{k-l+2}\=T_{k,l}\, .\]
 Therefore,~$T_{k,l}\in \Sym^\blok(\SA)$ for all~$k\geq l$. Moreover, by applying M\"obius inversion on Equation~(\ref{def:connectedprod}), which defines the connected product, we find
\begin{equation}\label{eq:grading}T_{\vec{k},\vec{l}} \= \sum_{\alpha\in \Pi(n)}\bigblok_{A\in \alpha}(T_{k_{A_1},l_{A_1}}\cp T_{k_{A_2},l_{A_1}}\cp \ldots). \end{equation}
As we already showed that~$T_{k,l}\in \Sym^\blok(\SA)$ if~$k\geq l$, we find~$T_{\vec{k},\vec{l}}\in \Sym^\blok(\SA)$ if~$k_i\geq l_i$ for all~$i$. 

Next, we show that all elements in~$\Sym^\blok(\SA)$ are a linear combination of the~$T_{\vec{k},\vec{l}}$ satisfying~${k_i\geq l_i}$. As~$\SA$ clearly is contained in the space generated by the~$T_{\vec{k},\vec{l}}$ for which~$k_i\geq l_i$, it suffices to show that the latter space is closed under~$\blok$. For this we follow the proof of \cref{thm:tclosed} observing that in each step~$k_i\geq l_i$, so that indeed the~$T_{\vec{k},\vec{l}}$ for which~$k_i\geq l_i$ form a generating set for~$\Sym^\blok(\SA)$. 

As we already showed that the~$T_{\vec{k},\vec{l}}$ are linearly independent, we conclude that the~$T_{\vec{k},\vec{l}}\in \TA^\mathrm{mon}$ satisfying~$k_i\geq l_i$ for all~$i$ form a basis for~$\Sym^\blok(\SA)$. 

The last part of the statement follows directly, as by definition all~$T_{\vec{k},\vec{l}}\in \TA^\mathrm{mon}$ satisfying~$l_i=1$ for all~$i$ generate~$\SA$, and by the above they are linearly independent. 
\end{proof}

\paragraph{The basis defining the weight grading}
By definition, the double moment functions generate~$\TA$ under the pointwise product. In this section we show that we can replace the pointwise product in the latter statement by the induced product. Again we will consider every reordering of the factors in~$T_{k_1,l_1}\blok \cdots \blok T_{k_n,l_n}$ due to commutativity of the products to be the same element. Then, we have:

\begin{thm}\label{thm:basis}
The elements~$T_{k_1,l_1}\blok \cdots \blok T_{k_n,l_n}$ form a basis for~$\TA$. A basis for the subspace~$\Sym^\blok(\SA)$ is given by the subset of elements for which~$k_i\geq l_i$ for all~$i$. 
\end{thm}
\begin{proof}
Assign to~$T_{k,l}$ weight~${k+l}$. This defines a weight filtering on~$\mathcal{T}$ with respect to the pointwise product. Consider the subspace of elements of weight at most~$w$ in~$\TA$. The number of basis elements in the basis given by the pointwise product in the previous section equals the number of induced products of the~$T_{k,l}$. Hence, it suffice that the induced products of the~$T_{k,l}$ generate~$\TA$. For this we proceed by induction first on the weight and then on the depth. Here, by depth we mean the unique filtering under the pointwise product for which every~$T_{k,l}$ has depth~$1$, usually called the \emph{total depth}. 

Trivially, every element of weight~$0$ or depth~$0$ is generated by (empty) induced products of the~$T_{k,l}$. Next, consider~${T_{\vec{k},\vec{l}}\in \TA}$ and assume all elements of lower weight and of the same weight and lower depth are generated by induced product of the~$T_{k,l}$. Let~$T_{\vec{k},\vec{l}}\in \TA$ of weight~$w$ be given and write~$\vec{k}',\vec{l}'$ for~$\vec{k},\vec{l}$ after omitting the last ($n$th) entry. Then,
\[T_{\vec{k},\vec{l}}\= T_{\vec{k}',\vec{l}'}\blok T_{k_n,l_n}   - T_{\vec{k}',\vec{l}'}\cp T_{k_n,l_n}\, .\]
Note that~$T_{\vec{k}',\vec{l}'}~$ is of weight strictly less than~$w$, hence is generated by induced products of the~$T_{k,l}$. Moreover, by \cref{prop:connectedbracketrecursion} and \cref{prop:computephi} it follows that the depth of ~$T_{\vec{k}',\vec{l}'}\cp T_{k_n,l_n}$ is at most~${n-1}$. Hence, by our induction hypothesis, it is generated by induced products of the~$T_{k,l}$. We conclude that~$T_{\vec{k},\vec{l}}$ is generated by induced products of the~$T_{k,l}$, which proves the first part of the theorem. 

The second part follows by the same proof, everywhere restricting to those~$T_{k,l}$ for which~${k\geq l}$. 
\end{proof}

By the above theorem, we can define a weight grading on~$\TA$.
\begin{defn} Define a \emph{weight grading} on~$\TA$ by assigning to~$T_{k,l}$ weight~$k+l$ and extending under the induced product. 
\end{defn}
Note that both the grading on~$\TA$ and the grading on~$\SA$ correspond to the grading on quasimodular forms after taking the~$q$-bracket. Hence, the grading on~$\SA$ is the restriction of the grading on~$\TA$. 

The weight grading defines a weight operator. In \cref{sec:sltwo} we extend this weight operator to an~$\sltwo$-triple acting on~$\TA$, so that~$\TA$ becomes an~$\sltwo$-algebra.

\paragraph{The~\texorpdfstring{$n$}{n}-point functions}
As induced products of the~$T_{k,l}$ form a basis for~$\TA$, knowing~$\langle f\rangle_q$ for all~$f\in \TA$ is equivalent to knowing the following generating function, called the~\emph{$n$-point function}
\[F_n(u_1,\ldots u_n, v_1,\ldots v_n) \= \sum_{\vec{k},\vec{l}} \langle T_{k_1,l_1}\blok \cdots \blok T_{k_n,l_n}\rangle_q \frac{u_1^{k_1}\cdots u_n^{k_n} v_1^{l_1-1} \cdots v_n^{l_n-1}}{k_1!\cdots k_n!(l_1-1)!\cdots (l_n-1)!}\]
for all~$n\geq 0$. Here the sum is over all~$k_i,l_i$ such that~$k_i+l_i$ is even and~$m!$ is consider to be~$1$ for~$m<0$. As the~$q$-bracket is a homomorphism with respect to the induced product, we directly conclude that
\begin{align}\label{eq:nmult}F_n(\vec{u},\vec{v})\=\prod_{i=1}^n F_1(u_i,v_i).\end{align}
We also define the \emph{partition function} by
\[\Phi(\vec{t}) \= \sum_{n=0}^\infty \frac{1}{n!} \sum_{\vec{k},\vec{l}}\langle T_{k_1,l_1}\blok \cdots \blok T_{k_n,l_n}\rangle_q t_{k_1,l_1}  \cdots  t_{k_n,l_n}\, .\]

The following result (together with~(\ref{eq:nmult})) expresses these functions in terms of the Jacobi theta series (see~(\ref{eq:jacobitheta})).
%\[\theta(z)=\sum_{\nu\in \z+\tfrac{1}{2}} (-1)^{\lfloor \nu\rfloor} e^{\nu z} q^{\nu^2/2}.\]
\begin{thm} For all~$n\geq 0$ one has
\[F_1(u,v) \= -\frac{1}{2}\frac{\theta'(0)\theta(u+v)}{\theta(u)\theta(v)}, \quad \quad \Phi(\vec{t}) \= \exp\Bigl([x^0y^0]F_1\left(\pdv{}{x},\pdv{}{y}\right)\sum_{k,l}t_{k,l} x^ky^l\Bigr),\]
where~$[x^0y^0]$ denotes taking the constant coefficient.
\end{thm}
\begin{proof}
We have that 
\[F_1(u,v) \= \prod_{i=1}^n \biggl(-\frac{1}{2u}-\frac{1}{2v}+\biggl(\sum_{k,l} D^{l-1} G_{k-l+2}+\sum_{k,l} D^{k} G_{l-k}\biggr)\frac{u^{k}v^{l-1}}{k!(l-1)!}\biggr),\]
where in the sum it is understood that~$k+l$ is even,~$k\geq 0, l\geq 1$. The expression for~$F_1(u,v)$ in the statement now follows from \cite[$\S\S 3$]{Zag91}. The expression for~$\Phi$ follows immediately from this result. 
\end{proof}

\section{Differential operators}\label{sec:sltwo}
\paragraph{The derivative of a function on partitions}\label{par:der}
Note that for all~$f\in \QtoP$ one has
\begin{equation}\label{eq:Dq-brac} D\langle f \rangle_q \=\langle S_2 f\rangle_q - \langle S_2\rangle_q\langle f \rangle_q\, .\end{equation}
Hence, by letting~$Df:=S_2\cp f=S_2f-S_2\blok f$ for~$f\in \QtoP$, we have that~$D\langle f \rangle_q = \langle Df \rangle_q\, .$ Moreover,~$D$ acts as a derivation:

\begin{prop}\label{prop:der}
The map~$D:\q^\partitions\to\q^\partitions$ is an equivariant derivation, i.e.~$D$ is linear, satisfies the Leibniz rule and 
\[D\langle f \rangle_q \= \langle Df\rangle_q.\]
\end{prop}

In fact, for all~$k\geq 1$, the mapping~$f\mapsto S_k\cp f$ is a derivation. Recall the definition of the M\"obius function~$\mu$ defined in \cref{prop:mobius} and denote $S_k^0=S_k-S_k(\emptyset)$. 

\begin{lem}For all even~$m\geq 2$ one has \begin{enumerate}[\upshape(i)]
\item\label{it:mobius1}~$\displaystyle S_m^0\blok \mu = -S_m^0\,\mu$;
\item The mapping~$(\q^\partitions,\blok)\to(\q^\partitions,\blok), f\mapsto S_m\cp f$ is a derivation, uniquely determined on~$\mathcal{T}$ by~\[\displaystyle S_m\cp T_{k,l} \= T_{k+m-1,l+1}\,.\]
\end{enumerate}
\end{lem}
\begin{remark}
In case~$m\geq 4$, the derivation~$f\mapsto S_m\cp f$ does not correspond to a derivation on~$\widetilde{M}$, i.e., a derivation~$\mathfrak{d}_m$ such that~$\mathfrak{d}_m\langle f\rangle_q = \langle S_m\cp f\rangle_q$ for all~$f\in \TA$. For instance, although the~$q$-brackets of~$T_{m,m}$ and~$T_{m-1,m+1}$ are the same, the~$q$-brackets of~$S_{m}\cp T_{m,m} = T_{2m-1,m+1}$ and~$S_{m}\cp T_{m-1,m+1} = T_{2m-2,m+2}$ are different. 
\end{remark}
\begin{proof}
First of all, by \cref{prop:phibrac} one has
\begin{align}\label{eq:der1}\Bigl(\,\sum_{\lambda \in \partitions}u_\lambda\Bigr)\langle S_k^0\blok\mu\rangle_{\vec{u}} \= \Bigl(\,\sum_{m,r\geq 1} m^{k-1}u_m^r\Bigr)\Bigl(\,\sum_{\lambda \in \partitions}\mu(\lambda)\,u_{\lambda}\Bigr).\end{align}
Let~$\mathscr{S}_m$ be the set of strict partitions not containing~$m$ as a part. Then, we can rewrite~(\ref{eq:der1}) as
\[ \sum_{m}\sum_{\lambda \in \mathscr{S}_m} m^{k-1}\mu(\lambda)\,u_mu_{\lambda}\= -\sum_{\lambda \in \partitions} S_k^0(\lambda)\,\mu(\lambda)\,u_{\lambda}\, ,\]
since~$\mu(\lambda\cup(m))=-\mu(\lambda)$ for~$\lambda \in \mathscr{S}_m$, so that for~$r\geq 2$ the coefficient of~$u_m^ru_\lambda$  cancels in pairs. We conclude that~$S_k^0\blok \mu = -S_k^0\,\mu$. 

For the second part, note that~(\ref{it:mobius1}) implies that 
\[S_k\blok \mu \= -\Bigl(S_k+\frac{B_k}{k}\Bigr)\,\mu.\] 
Let~$f,g\in \q^\partitions$ be given. Then 
\begin{align}
S_k\cp (f\blok g) &\= S_k(f\blok g) - S_k\blok f\blok g.
\end{align}
If~$\alpha\cup\beta\cup\gamma=\lambda$ then~$S_k(\lambda) = S_k(\alpha)+S_k(\beta)+S_k(\gamma)+\frac{B_k}{k}$, hence
\begin{align}
S_k(\lambda)\,(f\blok g)(\lambda) &\=\sum_{\alpha\cup\beta\cup\gamma=\lambda} \Bigl(S_k(\alpha)+S_k(\beta)+S_k(\gamma)+\frac{B_k}{k}\Bigr)f(\alpha)\,g(\beta)\,\mu(\gamma) \\
&\= (S_kf)\blok g \+ f\blok(S_kg) \+\sum_{\alpha\cup\beta\cup\gamma=\lambda} \Bigl(S_k(\gamma)+\frac{B_k}{k}\Bigr)f(\alpha)\,g(\beta)\,\mu(\gamma) \\
&\= (S_kf)\blok g \+ f\blok(S_kg) \meno \sum_{\alpha\cup\beta\cup\gamma=\lambda}(S_k\blok\mu)(\gamma)\,f(\alpha)\,g(\beta) \\
&\= (S_kf)\blok g+ f\blok(S_kg) \meno S_k\blok f\blok g. 
\end{align}
Therefore, 
\[S_k\cp (f\blok g) \= (S_kf)\blok g+ f\blok(S_kg)-2\,S_k\blok f\blok g \= (S_k\cp f)\blok g + f\blok(S_k\cp g),\]
i.e., the mapping~$f\mapsto S_k\cp f$ is a derivation.  The formula~$S_m\cp T_{k,l} = T_{k+m-1,l+1}$ follows directly from \cref{prop:computephi}.
\end{proof}

\begin{proof}[Proof of \cref{prop:der}]
As~$S_2\cp f = S_2f-S_2\blok f$ is derivation by the above lemma, the results follows directly from~(\ref{eq:Dq-brac}). 
\end{proof}

\paragraph{The equivariant~\texorpdfstring{$q$}{q}-bracket}\label{par:eq}
In this section we extend the action by the~$\sltwo$-triple~$(D, \mathfrak{d},W)$ on quasimodular forms to~$\TA$. As the derivation~$\mathfrak{d}$ does not act on all power series in~$q$, but only on quasimodular forms, we cannot hope to define~$\mathfrak{d}$ on all functions on partitions as we did with~$D$. On the algebra~$\TA$, however, this is possible. We define an~$\sltwo$-action on this space and we show that the~$q$-bracket restricted to~$\TA$ is an equivariant map of~$\sltwo$-algebras.

Note that the following definition agrees with the definition of~$D$ in the previous section:
\begin{defn}
Define the derivations~$D,W,\mathfrak{d}$ on~$\TA$ by
\begin{align}
D\, T_{k,l} &\= T_{k+1,l+1}\, , \\
W\, T_{k,l} &\=(k+l)T_{k,l}\, , \\
\mathfrak{d}\,T_{k,l} &\= k(l-1)T_{k-1,l-1}-\tfrac{1}{2}\delta_{k+l-2}.
\end{align}
\end{defn}
One immediately checks that~$D,W$ and~$\mathfrak{d}$ satisfy the commutation relation of an~$\sltwo$-triple on~$\TA$. The corresponding acting of~$\sltwo$ on~$\TA$ makes the~$q$-bracket equivariant, so that a refined version of \cref{thm:2} is:

\begin{thm}[The~$\sltwo$-equivariant symmetric Bloch--Okounkov theorem]\label{thm:sl2bo}
The algebra~$\TA$ is an~$\sltwo$-algebra with respect to the above action of~$\sltwo$ on~$\TA$. Moreover, the~$q$-bracket becomes an equivariant map of~$\sltwo$-algebras, i.e., for~$f \in \TA$ one has
\begin{equation}\label{eq:sl2q}
D\langle f \rangle_q = \langle Df\rangle_q, \quad W\langle f \rangle_q = \langle Wf\rangle_q, \quad \mathfrak{d}\langle f \rangle_q = \langle \mathfrak{d}f\rangle_q\,.
\end{equation}
\end{thm}

\begin{proof}
We already observed that the first of the three equality holds, and the second is the homogeneity statement. Hence, it suffices to prove the last statement. Using~(\ref{eq:D^n}) we find that for~$a\geq 0, b\geq 2$ one has 
\[\mathfrak{d}(D^aG_b) \= a(a+b-1)D^{a-1}G_b-\tfrac{1}{2}\delta_{a+b-2}.\]
Hence,
\[\mathfrak{d}\langle T_{k,l}\rangle_q \= k(l-1)\langle T_{k-1,l-1}\rangle_q-\delta_{k+l-2} \= \langle \mathfrak{d}\,T_{k,l}\rangle_q\]
 and the last statement follows from the Leibniz rule. 
\end{proof}

\paragraph{Rankin--Cohen brackets}\label{par:rc} 
The~$\sltwo$-action allows us to define Rankin--Cohen brackets on~$\TA$. 
\begin{defn} For two elements~$f,g\in \TA$ and $n\geq 0$ the~$n$th Rankin--Cohen bracket is given by
\begin{align}\label{defn:rc}[f,g]_n \= \sum_{\substack{r,s\geq 0\\r+s=n}}(-1)^r\binom{k+n-1}{s}\binom{l+n-1}{r}D^rf\blok D^sg.\end{align}
\end{defn}
Note that the formula~(\ref{defn:rc}) would have defined the Rankin--Cohen brackets on~$\widetilde{M}$ if~$D$ acts by~$q\pdv{}{q}$ and the induced product is replaced by the usual product, whereas in this line~$D$ acts on~$\TA$ as explained in the previous sections. 

If~$f,g\in \ker\mathfrak{d}$, then~$\langle f\rangle_q$ and~$\langle g \rangle_q$ are modular forms. The Rankin--Cohen bracket of two modular forms is a modular form; analogously we have:
\begin{prop}\label{prop:ker}
If~$f,g\in \ker\mathfrak{d}$, %i.e.~$\langle f\rangle_q$ and~$\langle g \rangle_q$ are modular forms, 
then~$[f,g]_n\in \ker\mathfrak{d}$. 
\end{prop}
\begin{proof}
Using~(\ref{eq:D^n}), we find that
\begin{align}
\mathfrak{d}[f,g]_n \= \sum_{\substack{r,s\geq 0\\r+s=n}}&(-1)^r\frac{(k+n-1)!}{s!(k+r-2)!}\frac{(l+n-1)!}{(r-1)!(k+s-1)!}D^{r-1}f\blok D^sg \+\\
&(-1)^r\frac{(k+n-1)!}{(s-1)!(k+r-1)!}\frac{(l+n-1)!}{r!(l+s-2)!}D^rf\blok D^{s-1}g,
\end{align}
where~$\frac{1}{(-1)!}$ should taken to be~$0$. This is a telescoping sum, vanishing identically. 
\end{proof}

\begin{remark}The above bracket makes the algebra~$\TA$ into a Rankin--Cohen algebra, meaning the following. Let~$A_*=\oplus_{k\geq 0} A_k$ be a graded~$K$-vector space with~$A_0=K$ and~$\dim A_k<\infty$ (for us~$A=\TA$). We say~$A$ is a \emph{Rankin--Cohen algebra} if there are bilinear operations~$[\,,\,]_n:A_k \otimes A_l \to A_{k+l+2n}\ (k, l, n> 0)$ which satisfy all the algebraic identities satisfied by the Rankin--Cohen brackets on~$\widetilde{M}$ \cite{Zag94}. %\comm{An unpublished result by Zagier and Nagatomo} states that there is a one-to-one correspondence between Rankin--Cohen algebras and~$\sltwo$-algebras. 
\end{remark}

\paragraph{A restricted~\texorpdfstring{$\sltwo$}{sl2}-action}
\Cref{thm:sl2bo} does not make~$\SA$ into an~$\sltwo$-algebra. Namely,~$D$ does not preserve~$\SA$. However, if we allow ourselves to deform the~$\mathfrak{sl}_2$-triple~$(D, \mathfrak{d}, W)$ as in \cite{vI18}, we can define an~$\sltwo$-action on~$\SA$. This action, however, does not make~$\SA$ into an~$\sltwo$-algebra, as the deformed operators are not derivations.

The operator taking the role of~$\mathfrak{d}$ is the operator~$\mathfrak{s}:\SA_k \to \SA_{k-2}$ defined by
\[\mathfrak{s} \= \frac{1}{2}\sum_{k,l\geq 0}(k+l) S_{k+l}\frac{\partial^2}{\partial S_{k+1}\,\partial S_{l+1}}-\frac{1}{2}\frac{\partial}{\partial S_2}.\]
The operator~$D$ is replaced by multiplication with~$S_2$.
\begin{lem}\label{lem:sl2}
The triple~$(S_2,\mathfrak{s} ,W-\tfrac{1}{2})$ forms an~$\sltwo$-triple of operators acting on~$\SA$.
\end{lem}
\begin{proof}
Observe that
\[[\mathfrak{s} ,S_2]f\=\sum_{k}(k+1)S_{k+1}\frac{\partial}{\partial S_{k+1}}f-\tfrac{1}{2}f\=(W-\tfrac{1}{2})f.\]
As~$\mathfrak{s}$ and~$S_2$ decrease, respectively, increase the weight by~$2$, the claim follows. 
\end{proof}

\begin{thm}\label{thm:sl2}
The~$q$-bracket~$\langle\, \cdot\, \rangle_q: \SA \to \quasimodular$ is an equivariant mapping with respect to the~$\sltwo$-triple~${(S_2, \mathfrak{s},W-\tfrac{1}{2})}$ on~$\SA$ and the~$\sltwo$-triple~$(D+G_2, \mathfrak{d}, W-\tfrac{1}{2})$ on~$\quasimodular$, i.e., for all~$f\in \SA$ one has
\begin{equation}\label{eq:sl2q-2}
(D+G_2)\langle f\rangle_q \= \langle S_2f\rangle_q, \quad (W-\tfrac{1}{2})\langle f\rangle_q\= \langle (E-\tfrac{1}{2})f\rangle_q, \quad \mathfrak{d}\langle f\rangle_q\=\langle \mathfrak{s}  f\rangle_q\,.
\end{equation}
\end{thm}
\begin{proof}
The first of the three equalities in~(\ref{eq:sl2q-2}) follows from the definition of the~$q$-bracket; the second is the homogeneity statement of Theorem~\ref{thm:sl2bo}. Hence, it remains to prove the last equation~$\mathfrak{d}\langle f\rangle_q=\langle \mathfrak{s}  f\rangle_q\, .$

Given~$\vec{k}\in \n^n$, let~$\vec{k}^{i}\in \n^{n-1}$ be given by~$\vec{k}^i:=(k_1,\ldots, k_{i-1},k_{i+1},\ldots, k_n)$ omitting~$k_i$. Similarly, define~$\vec{k}^{i,j}\in \n^{n-2}$ by omitting~$k_i$ and~$k_j$. Then
\[ \mathfrak{s}  S_{\vec{k}} \=\sum_{i\neq j}(k_i+k_j-2) S_{k_i+k_j-2} S_{\vec{k}^{i,j}} \meno \frac{1}{2}\sum_{i:\, k_i=2} S_{\vec{k}^{i}}\, .
\]
By \cref{Sisquasimodular} one finds
	\[\left\langle S_{k_i+k_j-2} S_{\vec{k}^{i,j}} \right\rangle_q \= \sum_{\substack{\beta\in \Pi(n)\\\exists I\in \beta:\, \{i,j\}\subset I}}\hspace{-15pt}D^{\ell(I)-2}\Eis_{|k_{I}|-2\ell(I)+2}\prod_{B\neq I} D^{\ell(A)-1}\Eis_{|k_B|-2\ell(A)+2}\,.\]
For~$I\in \beta$ and~$\vec{l}\in \n^I$, let
\[ C(I,\vec{l})\defis\sum_{i,j\in I, i\neq j} (l_i+l_j-2) \= (\ell(I)-1)(|\vec{l}|-\ell(I)).\]
It follows that~$\sum_{i\neq j} (k_i+k_j-2)\left\langle S_{k_i+k_j-2} S_{\vec{k}^{i,j}} \right\rangle_q~$ equals
\[\sum_{\beta\in \Pi(n)}\sum_{I\in \beta}2C(I,\vec{k})D^{\ell(I)-2}\Eis_{|k_{I}|-2\ell(I)+2}\prod_{B\neq I} D^{\ell(B)-1}\Eis_{|k_B|-2\ell(B)+2}\,.\]
On the other hand, observe that if~$f$ is of weight~$|\vec{l}|-2\ell(I)+2$, Equation~(\ref{eq:D^n}) yields
\[[\mathfrak{d},D^{\ell(I)-1}]f \= C(I,\vec{l})D^{\ell(I)-2}f.\]
 Hence, using~$\mathfrak{d}G_k \= -\tfrac{1}{2}\delta_{k,2}$, we obtain
 \[[\mathfrak{d},D^{\ell(I)-1}]G_{|k_{I}|-2\ell(I)+2} \= C(B,\vec{k}_I)D^{\ell(I)-2}G_{|k_{I}|-2\ell(I)+2} - \frac{1}{2}\delta_{\vec{k}_I,(2)}\,.\]
Therefore,
\begin{align}
\mathfrak{d}\langle S_{\vec{k}} \rangle_q &\= \sum_{\beta\in \Pi(n)} \sum_{I\in \beta} C(I,\vec{k}_I)D^{\ell(I)-2}\Eis_{|k_{I}|-2\ell(I)+2}\prod_{B\neq I} D^{\ell(B)-1}\Eis_{|k_B|-2\ell(B)+2} + \\
&\quad\quad -\frac{1}{2}\sum_{i:\, k_i=2} \sum_{\beta\in \Pi([n]\backslash\{i\})} \prod_{B\in \beta} D^{\ell(B)-1}\Eis_{|k_B|-2\ell(B)+2}\,,
\end{align}
which by the above reasoning is exactly equal to~$\langle \mathfrak{s} S_{\vec{k}}\rangle_q\,$. 
\end{proof}

\section{Relating the two products}\label{sec:2}
\paragraph{The structure constants}\label{par:coef}
In \cref{prop:computephi}, we deduced that 
\[T_{k_1,f_1}\cp \ldots\cp  T_{k_n,f_n}\=T_{|\vec{k}|,g} \quad \text{with} \quad g(f_1,\ldots,f_n)\=\sum_{\alpha\in\Pi(n)} \mob \partial^{\ell(\alpha)-1}\bigconv_{A\in \alpha} f_A.\]
In the particular case that~$f_1=\ldots=f_n$ is the identity function, we saw in \cref{lem:computephi} that~$g=\Faulhaber_n$. If~$f_1,\ldots,f_n$ are Faulhaber polynomials, the function~$g$ is not necessarily equal Faulhaber polynomial on all $m\in \n$, but, by \cref{lem:convispol},~$\partial g$ equals some polynomial. Also, using~$g$ is uniquely determined by~$\partial g$, the function~$g$ equals some polynomial. We expand~$g$ as a linear combination of Faulhaber polynomials. 

\begin{defn}\label{def:c}
Given integers~$l_1,\ldots, l_n$, we define \emph{the structure constants}~$C^{\vec{l}}_i$ by
\[g(\Faulhaber_{l_1},\ldots, \Faulhaber_{l_n}) \= \sum_{i=0}^{|\vec{l}|-1} C_{i}^{\vec{l}}\, \Faulhaber_{|\vec{l}|-i}\, .\]
\end{defn}
Observe that~$C_{i}^{\vec{l}}=0$ for odd~$i$, as~$\partial g$ is even or odd. \Cref{lem:computephi} is the statement 
\[C_{i}^{(1,\ldots,1)} \= \begin{cases} 1 & i=0 \\ 0 & \text{else.} \end{cases}\]
More generally, by \cref{prop:computephi}(\ref{prop:computephi-i}) one has~$C^{1,\vec{l}}_i=C^{\vec{l}}_i$, so that w.l.o.g.\ we can assume~$l_i>1$. 
In this section, we give an explicit, but involved, formula for these coefficients in terms of Bernoulli numbers and binomial coefficients. In order to do so, for~$l_1,l_2\geq 1$ and~$i\in \z_{\geq 0}$, we introduce the following numbers:
\[\coef_i^{l_1,l_2} \defis \begin{cases} \frac{(l_1-1)!(l_2-1)!}{(l_1+l_2-1)!} & i=0, \\ \zeta(1-i)\left((-1)^{l_2}\binom{l_1-1}{i-l_2}+(-1)^{l_1}\binom{l_2-1}{i-l_1}\right) & i>0,\end{cases}\]
which by \cite[Proposition~A.10]{AIK14} satisfy
\[\sum_{i=0}^{l_1+l_2-2}\coef_i^{l_1,l_2}\frac{B_{l_1+l_2-i}}{l_1+l_2-i} = (-1)^{l_1l_2}\frac{B_{l_1+l_2}-B_{l_1}B_{l_2}}{l_1 l_2}.\]
Note that~$\zeta(1-i)=(-1)^{i+1}\frac{B_i}{i}$ for~$i\geq 1$. 
The following polynomials can be expressed in terms of these coefficients:
\begin{lem} For all~$l_1,l_2,\ldots,l_r\geq 2$ one has the following identities:
\begin{enumerate}[\upshape(i)]
\item $\displaystyle\Faulhaber_{l_1}(x) \= \sum_{i=0}^{\infty} \coef_i^{l_1,1} x^{l_1-i}$;
\item\label{it:fhii} $\displaystyle(\partial\Faulhaber_{l_1}*\partial\Faulhaber_{l_2})(x) \= \sum_{i=0}^\infty \coef_i^{l_1,l_2}x^{l_1+l_2-i-1}$;
\item $\displaystyle\partial(\Faulhaber_{l_1}\cdots\Faulhaber_{l_r})(x) \= 2\sum_{|\vec{i}|\equiv 1\, (2)} \mathpzc{B}_{i_1}^{l_1,1}\cdots \mathpzc{B}_{i_r}^{l_r,1}x^{|\vec{l}|-|\vec{i}|}$.
\end{enumerate}
\end{lem}
\begin{proof}
The first two equations, of which the former is the well-known expansion of the Faulhaber polynomials, follow by considering the corresponding generating series. In order to prove~(\ref{it:fhii}), we let~$n\in \n$ and consider
\begin{align}
\mathcal{G}(n)\defis&\sum_{l_1,l_2=1}^\infty(\partial\Faulhaber_{l_1}*\partial\Faulhaber_{l_2})(n) \frac{z_1^{l_1-1}}{(l_1-1)!} \frac{z_2^{l_2-1}}{(l_2-1)!} \\
\=&\sum_{m_1+m_2=n}e^{m_1z_1+m_2z_2}\\ 
%\=& \frac{e^{nz_1+z_2}-e^{z_1+nz_2}}{e^{z_1}-e^{z_2}} \\
\=& \frac{e^{nz_1}}{e^{z_1-z_2}-1}+\frac{e^{nz_2}}{e^{z_2-z_1}-1}.
\end{align}
As the generating series of the Bernoulli numbers~$\sum_{j=0}^\infty B_j \frac{z^j}{j!} = z(e^z-1)^{-1}$ implies that
\begin{align}
\frac{1}{e^{z_1-z_2}-1} %=\sum_{j=0}^\infty\frac{B_j}{j!}(z_1-z_2)^{j-1}
\=\frac{1}{z_1-z_2} + \sum_{j=1}^\infty\sum_{i=0}^{j-1} \frac{B_j}{j}(-1)^{i} \frac{z_1^{j-i-1}z_2^{i}}{(j-1-i)!i!},\end{align}
we find 
\begin{align}
\mathcal{G}(n) 
%&\= \frac{e^{nz_1}-e^{nz_2}}{z_1-z_2}+\sum_{j=1}^\infty\sum_{i=0}^{j-1} \frac{B_j}{j}(-1)^{j-i} \left(\frac{z_1^{j-i-1}%z_2^{i}}{(j-i-1)!i!}e^{z_1n}-\frac{z_1^iz_2^{j-i-1}}{i!(j-i-1)!}e^{z_2n}\right) \\ 
&\=\sum_{k=1}^\infty\biggl(\,\sum_{i=0}^{k-1}\frac{z_1^iz_2^{k-i-1}}{i!(k-i-1)!}+\sum_{j=1}^\infty\sum_{i=0}^{j-1} \frac{B_j}{j}(-1)^{i} \biggl(\frac{z_1^{k+j-i-1}z_2^{i}}{(j-i-1)!i!}+\frac{z_1^iz_2^{k+j-i-1}}{i!(j-i-1)!}\biggr)\biggr)\frac{n^k}{k!} \\
&\= \sum_{l_1,l_2=1}^\infty \sum_{i=0}^\infty (-1)^i\coef_{i}^{l_1,l_2} \frac{z_1^{l_1-1}}{(l_1-1)!} \frac{z_2^{l_2-1}}{(l_2-1)!} n^{l_1+l_2-i-1}.
\end{align}
Since~$\coef_{i}^{l_1,l_2}$ vanishes for odd~$i$ if~$l_1,l_2>1$, this proves the second equation.
The third equation follows from the first by noting that 
\begin{align}\partial(\Faulhaber_{l_1}\cdots\Faulhaber_{l_r})(x)&\= (\Faulhaber_{l_1}\cdots\Faulhaber_{l_r})(x)-(-1)^{|\vec{l}|}(\Faulhaber_{l_1}\cdots\Faulhaber_{l_r})(-x).\qedhere\end{align}
\end{proof}
Using these identities, one obtains
 \[ C_i^l \= \coef^{1,1}_{i}=\delta_{i,0}, \quad C_i^{l_1,l_2} \= \coef_i^{l_1,1}+\coef_i^{l_2,1}-\coef_{i}^{l_1,l_2} \]
These easy expressions for small~$n$ are misleading, as~$6C_i^{l_1,l_2,l_3}$ equals
\[\frac{1}{4}\delta_{i,2} \+ 3\sum_{\substack{i_1,i_2 \equiv 0\, (2) \\ i_1+i_2=i}} \coef_{i_1}^{l_{1},1}\coef_{i_2}^{l_{2},1} \meno \sum_{\substack{i_1 \equiv 1\, (2), j_1 \\i_1+j_1=i}} \coef_{i_1}^{l_{1},1}\coef_{j_1}^{l_{1},l_{2}+l_{3}-i_1} \+ 2\sum_{j_1+j_2=i}\coef_{j_1}^{l_{1},l_{2}}\coef_{j_2}^{l_{1}+l_{2}-j_1,l_{3}}\]
up to full symmetrization, i.e., summing over all~$\sigma \in S_3$ with~$l_i$ replaced by~$l_{\sigma(i)}$. 
%\[ C_i^{l_1,l_2,l_3} =\frac{1}{3!}\sum_{\sigma \in S_3}\left(\frac{1}{4}\delta_{i,2}+ 3\sum_{i_1,i_2 \equiv 0\, (2)} \coef_{i_1}^{l_{\sigma(1)},1}\coef_{i_2}^{l_{\sigma(2)},1} - \sum_{i_1 \equiv 1\, (2), j_1} \coef_{i_1}^{l_{\sigma(1)},1}\coef_{j_1}^{l_{\sigma(1)},l_{\sigma(2)}+l_{\sigma(3)}-i_1}+2\sum_{j_1,j_2}\coef_{j_1}^{l_{\sigma(1)},l_{\sigma(2)}}\coef_{j_2}^{l_{\sigma(1)}+l_{\sigma(2)},l_{\sigma(3)}}\right)\]
In general, given~$\alpha\in \Pi(n)$, write~$\alpha=\{A_1,\ldots, A_r\}$ and denote~$A^j=\cup_{i=1}^j A_j$. Also, for a vector~$\vec{k}$  and a set~$B$ we let~$k_B=\sum_{b\in B} k_b$. Then, the above observations allows us to write down the following formula, which is very amenable to computer calculation:
\begin{prop}\label{prop:structureconstants} Let~$l_1,\ldots, l_n>1$. Then
\[C_i^{\vec{l}} \= \sum_{\alpha \in \Pi(n)}\!2^{r}\mu(\alpha,\mathbf{1}) \!\sum_{\substack{i_1,\ldots,i_n\\ |\vec{i}_A|\equiv 1 \, (2)}}\left(\prod_{k=1}^n \coef_{i_k}^{l_k,1}\right)\!
 \left(\sum_{\substack{j_1,\ldots,j_{\ell(\alpha)-1}\\|\vec{i}|+|\vec{j}|=i+r}} \prod_{s=1}^{r-1} \coef^{l_{A^s}-j_{s-1},l_{A_{s+1}}-i_{A_{s+1}}+1}_{j_s}\right)\!\raisebox{-20pt}{.}\]
Here,~$j_0:=l_{A_0}-i_{A_0}$. 
\end{prop}
Note that the latter formula is written in an asymmetric way, but (by associativity of the convolution product) is symmetric in the~$l_i$. 

\paragraph{From the pointwise product to the induced product}\label{par:Tklf}
Suppose an element of~$\TA$ is given, written in the basis with respect to the pointwise product. How do we determine its (possibly mixed) weight and its representation in terms of the basis with respect to the~$\blok$ product? A first answer is given by applying M\"obius inversion to Equation~(\ref{def:connectedprod}), as given by Equation~(\ref{eq:grading}), i.e.,
\begin{equation}\label{eq:grading2} T_{\vec{k},\vec{l}} \= \sum_{\alpha\in \Pi(n)}\bigblok_{A\in \alpha}(T_{k_{A_1},l_{A_1}}\cp T_{k_{A_2},l_{A_2}}\cp \ldots). \end{equation}
However, as every factor~$T_{k_{A_1},l_{A_1}}\cp T_{k_{A_2},l_{A_2}}\cp \ldots$ in the above equation is a linear combination of generators of different weights, it is useful to have a recursive version of this result. For this, we write~$\pdv{}{T_{k,l}}$ for the derivative of~$f\in\TA$ in the former basis (with respect to the pointwise product) and~$\pdv{}{T_{\vec{k},\vec{l}}}$ for~$\prod_i \pdv{}{T_{k_i,l_i}}$.

\begin{prop}Let~$k,l\geq 1$. There exist differential operators~$\mathfrak{s}_{i,j}$ for all~$i,j\in \mathbb{Z}$  such that~$\mathfrak{s}_{i,j}=0$ if~$j<0$ and for all~$f\in \TA$ one has
\[T_{k,l}f \= \sum_{i\geq 0}\sum_{j\geq -l+1} T_{k+i,l+j} \blok \mathfrak{s}_{i,j}(f).\]
Explicitly, 
\[
\mathfrak{s}_{i,j} \= \sum_{|\vec{a}|=i} \mathfrak{t}_{\vec{a},j}, \quad\quad 
\mathfrak{t}_{\vec{a},j} \= \sum_{\vec{b}} C^{l,\vec{b}}_{|\vec{b}|-j} \pdv{}{T_{\vec{a},\vec{b}}},
\]
where~$\vec{a}$ and~$\vec{b}$ are vectors of integers of the same length and with $|\vec{a}|=i$, the structure constants~$C^{l,\vec{b}}_{|\vec{b}|-j}$ are as in \cref{prop:structureconstants} and~${l,\vec{b}}$ denotes the vector $(l,b_1,b_2,\ldots)$.
\end{prop}
\begin{proof}
By linearity, it suffices to prove the statement for monomials~$T_{\vec{k},\vec{l}}$. Hence, assume~$f=T_{\vec{k},\vec{l}}$. Applying  ~(\ref{eq:grading2}), extracting the factor containing~$T_{k,l}$ and applying~(\ref{eq:grading2}) again, yields
\begin{align}
T_{k,l}f&\=\sum_{A\subset [n]}(T_{k,l}\cp T_{k_{A_1},l_{A_1}}\cp T_{k_{A_2},l_{A_2}}\cp \ldots)\odot T_{\vec{k}_{[n]\backslash A},\vec{l}_{[n]\backslash A}} \\
&\=\sum_{\vec{a},\vec{b}}(T_{k,l}\cp T_{a_1,b_1}\cp T_{a_2,b_2}\cp \ldots)\odot \pdv{}{T_{\vec{a},\vec{b}}} f.
\end{align}
By \cref{def:c} this equals
\begin{align}
T_{k,l}f &\= \sum_{\vec{a},\vec{b}}\sum_{j\in \z} C_{j}^{l,\vec{b}} T_{|\vec{a}|+k,|\vec{b}|+l-j} \odot\pdv{}{T_{\vec{a},\vec{b}}} f
\end{align}
Replacing~$j$ by~$-j+|\vec{b}|$ and writing~$i=|\vec{a}|$, one obtains
\begin{align}
T_{k,l}f &\= \sum_{i\geq0}\sum_{j\in \z}  T_{k+i,l+j} \odot\sum_{|\vec{a}|=i}\sum_{\vec{b}}C_{|\vec{b}|-j}^{l,\vec{b}}\pdv{}{T_{\vec{a},\vec{b}}} f,	
\end{align}
as desired. 
\end{proof}

\begin{cor} For all~$k,l\geq 1$ and~$f\in \TA$ one has
\[\langle T_{k,l}f \rangle_q \= \sum_{a\geq 0}\sum_{b\geq 2} D^{a}G_{b} \langle\mathfrak{T}_{k,l}^{a,b}f\rangle_q\, ,\]
where~$\mathfrak{T}_{k,l}^{a,b}=\mathfrak{s}_{a-l+1,a+b-k-1}+\mathfrak{s}_{a+b-l,a-k}$. 
\end{cor}
\begin{proof} Distinguishing two cases in the previous result yields
\begin{align}
 \langle T_{k,l}f\rangle_q &\=\sum_{j< k+i-l} D^{l+j-1}G_{k+i-l-j+2} \blok \mathfrak{s}_{i,j}(f) \+  \sum_{i\geq 0}\sum_{j\geq k+i-l} D^{k+i}G_{l+j-k-i} \blok \mathfrak{s}_{i,j}(f)\\
&\= \sum_{a\geq 0}\sum_{b\geq 2} D^{a}G_{b} \langle(\mathfrak{s}_{a+b-k-1,a-l+1}+\mathfrak{s}_{a-k,a+b-l})(f)\rangle_q\,. \qedhere
\end{align}
\end{proof}

\section{Related functions on partitions}\label{sec:application}
We apply our results to interesting functions on partitions. 

\paragraph{Hook-length moments}\label{sec:hlm}
First of all, we focus on the hook-length moments~$H_k$ \cite[part III]{CMZ16}. These functions form a bridge between the symmetric algebra studied in this note and the shifted symmetric functions: the~$H_k$ themselves are shifted symmetric as
\begin{align}\label{eq:hooks} H_k(\lambda) \= \frac{1}{2}\sum_{i=0}^{k} \binom{k-2}{i-1}(-1)^i Q_i(\lambda)Q_{k-i}(\lambda).\end{align}
and they are also equal to the M\"oller transform of the symmetric~$S_k$, i.e.~$H_k=\moller(S_k)$, meaning the following. Denote~$z_\nu = \frac{n!}{|C_\nu|}$ with~$|C_\nu|$ the size of the conjugacy class corresponding to~$\nu$. Recall that
\[z_\nu \= \prod_{m=1}^\infty m^{r_m(\nu)}r_m(\nu)!.\]
 Given~$f\in \QtoP$, \emph{the M\"oller transform} of~$f$ at a partition~$\lambda\in \partitions(n)$ is given by \cite[Eqn~(45)]{Zag16}
\[\moller(f)(\lambda) \= \sum_{\nu\vdash n} z_\nu^{-1}\chi^\lambda(\nu)^2 f(\nu),\]
where the sum~$\nu\vdash n$ is over all partitions of size~$n$ and~$\chi^\lambda(\rho)$ denotes the the character of the representation corresponding to the partition~$\lambda$ evaluated at the conjugacy class corresponding to~$\rho$. Then,~$\langle\mathcal{M}(f)\rangle_q$ is a quasimodular form if and only if~$\langle f \rangle_q$ is a quasimodular form (which follows directly by the column orthogonality relations for the symmetric group). In the next section we study the M\"oller transform of elements of~$\TA$, but first, we explain the Murnaghan--Nakayama rule, used in \cite[part III]{CMZ16} to show equality between~$\moller(S_k)$ and~(\ref{eq:hooks}) and give two other expressions for the hook-length moments.

To start with the latter, \emph{the hook-length moments},  as their name suggests,  are defined as moments of the hook-lengths, i.e.
\[ H_k(\lambda) \= -\frac{B_k}{2k} + \sum_{\xi \in Y_\lambda} h(\xi)^{k-2},\]
where~$Y_\lambda$ denotes the Young diagram of a partition~$\lambda$ and~$h(\xi)$ denotes the hook-length of a cell~$\xi\in Y_\lambda$.

Next, the following constructions related to the Young diagram, give rise to the Murnaghan--Nakayama rule for the characters of the symmetric group. Given partitions~$\lambda,\nu$ with~$\nu_i\leq \lambda_i$ for all~$i$, we define the \emph{skew Young diagram}~$\lambda/\nu$ by removing the cells of~$Y_\nu$ from~$Y_\lambda$. Denote by~$|\lambda/\nu|=|\lambda|-|\nu|$ the number of cells of this diagram. We call~$\lambda/\nu$ a \emph{border strip} of~$\lambda$ if it is connected (through edges, not only through vertices) and contains no~$2\times2$-block. If~$\gamma=\lambda/\nu$ we write~$\lambda\setminus \gamma$ for~$\nu$. The \emph{height} of a border strip~$\gamma$ is defined to be one less than the number of columns and denoted by~$\height(\gamma)$.  
Given~$\vec{m}\in \n^s$, we let a \emph{border strip tableau}~$\gamma$ of type~$\vec{m}$ be a sequence~$\gamma_1,\ldots, \gamma_s$ such that~$\gamma_i$ is a border strip of~$\lambda\smallsetminus \gamma_1\smallsetminus\cdots \smallsetminus \gamma_{i-1}$ and~$|\gamma_i|=m_i$. Write~$Y_\gamma$ for the skew Young diagram consisting of all boxes of all the~$\gamma_i$ and write~$\height(\gamma)=\height(\gamma_1)+\ldots+\height(\gamma_s)$. Denote by~$\BST(\lambda,\vec{m})$ and~$\BST(\lambda/\nu,\vec{m})$ the set of all border strip tableau of type~$\vec{m}$ within~$\lambda$ and~$\lambda/\nu$, respectively. 

\begin{figure}[h!]
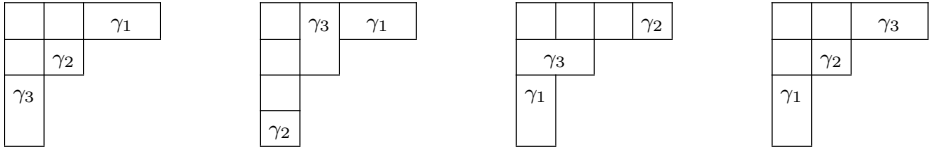
\begin{centering}{\scriptsize
\renewcommand{\arraystretch}{1.42}
\setlength\tabcolsep{3pt}
\begin{minipage}[h]{0.20\textwidth}
\begin{tabular}{|llll}
\hline
\multicolumn{1}{|l|}{} & \multicolumn{1}{l|}{}  & \multicolumn{2}{c|}{$\gamma_1$} \\ \cline{1-4} 
\multicolumn{1}{|l|}{} &  \multicolumn{1}{c|}{$\gamma_2$} \\\cline{1-2}
\multicolumn{1}{|l|}{$\gamma_3$} \\
\multicolumn{1}{|l|}{} & \multicolumn{1}{l}{\phantom{$\gamma_1$}} & \multicolumn{1}{l}{\phantom{$\gamma_1$}} & \multicolumn{1}{l}{\phantom{$\gamma_1$}} \\\cline{1-1}
\end{tabular}
\end{minipage}
\begin{minipage}[h]{0.20\textwidth}
\begin{tabular}{|llll}
\hline
\multicolumn{1}{|l|}{} & \multicolumn{1}{l|}{$\gamma_3$}  & \multicolumn{2}{c|}{$\gamma_1$} \\ \cline{1-1}\cline{3-4} 
\multicolumn{1}{|l|}{} &  \multicolumn{1}{c|}{} \\\cline{1-2}
\multicolumn{1}{|l|}{} \\\cline{1-1}
\multicolumn{1}{|l|}{$\gamma_2$} & \multicolumn{1}{l}{\phantom{$\gamma_1$}} & \multicolumn{1}{l}{\phantom{$\gamma_1$}} & \multicolumn{1}{l}{\phantom{$\gamma_1$}} \\\cline{1-1}
\end{tabular}
\end{minipage}
\begin{minipage}[h]{0.20\textwidth}
\begin{tabular}{|llll}
\hline
\multicolumn{1}{|l|}{} & \multicolumn{1}{l|}{}  & \multicolumn{1}{l|}{}& \multicolumn{1}{l|}{$\gamma_2$} \\ \cline{1-4} 
\multicolumn{2}{|c|}{$\gamma_3$} \\\cline{1-2}
\multicolumn{1}{|l|}{$\gamma_1$} \\
\multicolumn{1}{|l|}{} & \multicolumn{1}{l}{\phantom{$\gamma_1$}} & \multicolumn{1}{l}{\phantom{$\gamma_1$}} & \multicolumn{1}{l}{\phantom{$\gamma_1$}} \\\cline{1-1}
\end{tabular}
\end{minipage}
\begin{minipage}[h]{0.20\textwidth}
\begin{tabular}{|llll}
\hline
\multicolumn{1}{|l|}{} & \multicolumn{1}{l|}{}  & \multicolumn{2}{c|}{$\gamma_3$} \\ \cline{1-4} 
\multicolumn{1}{|l|}{} &  \multicolumn{1}{c|}{$\gamma_2$} \\\cline{1-2}
\multicolumn{1}{|l|}{$\gamma_1$} \\
\multicolumn{1}{|l|}{} & \multicolumn{1}{l}{\phantom{$\gamma_1$}} & \multicolumn{1}{l}{\phantom{$\gamma_1$}} & \multicolumn{1}{l}{\phantom{$\gamma_1$}} \\\cline{1-1}
\end{tabular}
\end{minipage}}
\caption{The Young diagrams corresponding to the border strip tableaux of type~$(2,1,2)$ within~${\lambda=(4,2,1,1)}$.}
\end{centering}
\end{figure}
\renewcommand{\arraystretch}{1}

The \emph{Murnaghan--Nakayama rule} (recursively) expresses the characters of the symmetric groups in terms the heights of border strip tableau.  Namely, if~$\rho'\subseteq \rho$ (both~$\rho'$ and~$\rho$ considered as multisets)
\[\chi^\lambda(\rho) \= \sum_{\gamma \in \BST(\lambda,\rho')} (-1)^{\height(\gamma)} \chi^{\lambda\backslash \gamma}(\rho-\rho'),\]
where~$\rho-\rho'$ denotes the difference of (multi)sets.
Of particular interest are the cases~$\rho'=\rho$ and~$\rho'=(\rho_1)$, yielding a direct or recursive combinatorial formula for~$\chi^\lambda(\rho)$ respectively:
\[\chi^\lambda(\rho)\=\sum_{\gamma \in \BST(\lambda,\rho)} (-1)^{\height(\gamma)} \qquad\text{and}\qquad \chi^\lambda(\rho) = \sum_{|\gamma|=\rho_1} (-1)^{\height(\gamma)} \,\chi^{\lambda\backslash \gamma}(\rho_2,\rho_3,\ldots),
\]
where the latter sum is over all borders strips~$\gamma$ of~$\lambda$ of length~$\rho_1$. The \emph{skew character}~$\chi^{\lambda/\nu}(\rho')$ is defined by ($|\lambda/\nu|=|\rho'|$)
\begin{align}\label{def:skewchar} \chi^{\lambda/\nu}(\rho')\= \sum_{\substack{\gamma \in \BST(\lambda/\nu,\rho')}} (-1)^{\height(\gamma)}, \end{align}
so that
\[\chi^\lambda(\rho) \= \sum_{|\nu|=|\rho'|}\chi^{\lambda/\nu}(\rho')\,\chi^{\nu}(\rho-\rho').\]
To conclude, we have the following definitions of the hook-length moments:
\begin{defn}\label{defn:hlm} The hook-length moments~$H_k$ ($k\geq 2$ even) are defined by either of the following equivalent definitions \cite[Section 13]{CMZ16}:
\begin{enumerate}[\upshape (i)]
\item $\displaystyle H_k(\lambda) \= -\frac{B_k}{2k} + \sum_{\xi \in Y_\lambda} h(\xi)^{k-2}\,;$
\item\label{it:X} $\displaystyle H_k(\lambda) \= -\frac{B_k}{2k} + \sum_{m=1}^\infty |\!\BST(\lambda,m)|\, m^{k-2}\,;$
\item $\displaystyle H_k \= \frac{1}{2}\sum_{i=0}^{k} \binom{k-2}{i-1}(-1)^i Q_i\,Q_{k-i}\,;$
\item $\displaystyle H_k \= \moller(S_k).$
\end{enumerate}
\end{defn}

\paragraph{Border strip moments}
The hook-length moments are M\"oller transformations of the~$S_k$. In this section we study the M\"oller transformation of the algebra~$\TA$, which contains the vector space spanned by all the~$S_k$. In order to do so, we express elements of~$\TA$ in terms of functions~$U_{\vec{k},\vec{l}}$ for which the induced product and M\"oller transformation are easy to compute. However, these function do not admit the property that the~$q$-bracket is quasimodular if~$k_i+l_i$ is even for all~$i$: each~$U_{\vec{k},\vec{l}}$ lies in the space generated by all the~$T_{\vec{k},\vec{l}}$ (possibly with~$k_i+l_i$ odd). 

Let 
\[\n(\vec{l}) \= \{(\underbrace{m_1,\ldots, m_1}_{l_1},\underbrace{m_2,\ldots,m_2}_{l_2},\ldots) \mid m_i\geq 1\}\]
the set of tuples of~$n:=|\vec{l}|$ positive integers, where the first~$l_1$, the second~$l_2$, etc. integers agree. 
For $\vec{k}\in \z_{\geq 0}^n$, define
\[U_{\vec{k},\vec{l}} \= \sum_{\vec{m}\in \n(\vec{l})} \vec{m}^{\vec{k}}\prod_{a=1}^\infty \binom{r_a(\lambda)}{r_a(\vec{m})}.\]
Observe that this product converges since~$r_a(\vec{m})=0$ for all but finitely many values of~$a$. Let~$\mathcal{U}$ be the algebra generated by the~$U_{\vec{k},\vec{l}}$. 

Generalise the hook-length moments in \cref{defn:hlm}(\ref{it:X}) by the following notion:
\begin{defn} The \emph{border strip moments} are given by
\[X_{\vec{k},\vec{l}}(\lambda) \= \sum_{\vec{m}\in \n(\vec{l})}\sum_{\gamma \in \BST(\lambda,\vec{m})} \frac{\chi^{\gamma}(\vec{m})^2}{z_{\vec{m}}}\, \vec{m}^{\vec{k}}\, .\]
Let~$\mathcal{X}$ be the vector space spanned by all the~$X_{\vec{k},\vec{l}}$. Define a filtration on~$\mathcal{X}$ by assigning to~$X_{\vec{k},\vec{l}}$ degree~$|\vec{k}|+|\vec{l}|$. 
\end{defn}
\begin{remark}Observe that for~$n=1$ and~$l=1$, the sum restricts to a sum over all border strips~$\gamma$ of~$\lambda$ and for such a border strip~$\gamma$ the factor~$\chi^{\gamma}(\vec{m})^2$ equals~$1$ and~$z_m$ equals~$m$. As the set of hook-lengths is in bijection with the set of all border strip lengths, one has that~$-\frac{B_k}{2k}+X_{k,1} = H_{k+1}$. 
\end{remark}%Moreover, by~(\ref{def:skewchar}) we can write
%\[X_{\vec{k},\vec{l}}(\lambda) = \sum_{\nu\in \partitions}\sum_{\gamma,\gamma' \in \mathrm{BS}_{\vec{l}}(\lambda\backslash \nu)} \frac{(-1)^{\height\gamma+\height\gamma'}}{z_{\Gamma}} |\gamma_1|^{k_1}\cdots |\gamma_n|^{k_n},\]
%where~$\mathrm{BS}_{\vec{l}}(\lambda\backslash \nu)$ denotes the set of all borderstrip tableau~$\gamma_{1,1},\ldots, \gamma_{1,l_1},\gamma_{2,1},\ldots, \gamma_{2,l_2},\ldots, \gamma_{n,l_n}$ of~$\lambda\backslash\nu$ for which~$|\gamma_{i,1}|=|\gamma_{i,j}|$ for all~$i,j$ and~$\Gamma$ is the vector consisting of the~$|\gamma_{i,j}|$. 

Denote by~$\stirling{n}{j}$ the Stirling numbers of the second kind (i.e., the number of elements in~$\Pi(n)$ of length~$j$). 
\begin{prop}\label{prop:Hkl} For all~$k\geq 0, l\geq 1, \vec{k},\vec{k}'\in \z_{\geq 0}^n,$ and integer vectors $\vec{l},\vec{l}'$ with $|\vec{l}|=|\vec{l}'|=n,$ one has
\begin{enumerate}[\upshape(i)]
\item\label{it:H1}~$\displaystyle T_{k,l} \= -\frac{B_{k+l}}{2(k+l)}(\delta_{l,1}+\delta_{k,0})+\sum_{j=1}^l \stirling{l}{j}(j-1)! \,U_{k,j}\, ;$
\item\label{it:H2}~$\displaystyle U_{\vec{k},\vec{l}}\blok U_{\vec{k}',\vec{l}'} \= U_{\vec{k}\cup \vec{k}',\vec{l}\cup \vec{l}'}\, ;$
\item\label{it:H3}~$\displaystyle \moller(U_{\vec{k},\vec{l}}) \= X_{\vec{k},\vec{l}}\, .$
\end{enumerate}
\end{prop}
\begin{proof}
For the first property, we use the known identity
\[x^{l-1} \= \sum_{j=1}^{l} \stirling{l}{j}(j-1)!\binom{x-1}{j-1}.\]
As~$\Faulhaber_l(x)$ and~$\binom{x}{j}$ are the unique polynomials with constant term equal to zero and such that~$\partial \Faulhaber_l(x)=x^{l-1}$ and~$\partial \binom{x}{j} = \binom{x-1}{j-1}$ respectively, we find
\[\Faulhaber_l(x) \=  \sum_{j=1}^{l} \stirling{l}{j}(j-1)!\binom{x}{j},\]
which yields property~(\ref{it:H1}). 	

Next, we show that for all~$i,j\geq 0$ one has
\[\partial\binom{x}{i} \conv \binom{x}{j} \= \binom{x}{i+j}.\]
%Observe that
%\[\partial \binom{x}{i}=\binom{x}{i}-\binom{x-1}{i}=\binom{x-1}{i-1}.\]
Both
\[\binom{x}{i} \conv \binom{x}{j} \= \sum_{m=0}^{x}\binom{m}{i} \binom{x-m}{j} \quad\quad\text{and}\quad\quad \binom{x+1}{i+j+1}\]
are polynomials of degree at most~$i+j+1$ taking the value~$0$ for~$x=0,1,\ldots,i+j-1$ and the value~$1$ for~$x=i+j$; hence, they are equal. Therefore,
\[\partial \binom{x}{i} \conv \binom{x}{j} \= \partial \binom{x+1}{i+j+1}=\binom{x}{i+j}.\]
By \cref{lem:disjointsupport} property~(\ref{it:H2}) follows. 

Finally, we have that
\begin{align}\label{eq:mollerofh}\moller(U_{\vec{k},\vec{l}})(\lambda) \= \sum_{\vec{m}\in \n(\vec{l})}\vec{m}^{\vec{k}} \sum_{\nu\vdash n}z_\nu^{-1}\chi^\lambda(\nu)^2  \prod_{a=1}^\infty \binom{r_a(\nu)}{r_a(\vec{m})}.\end{align}
Observe that given~$\vec{m}$ and~$\nu$ the term
\begin{align}\label{eq:mollerterm} z_\nu^{-1}\chi^\lambda(\nu)^2  \prod_{a=1}^\infty \binom{r_a(\nu)}{r_a(\vec{m})}\end{align}
 vanishes unless~$r_{a}(\nu)\geq r_a(\vec{m})$ for all positive~$a$. Let~$\nu'$ be the partition obtained from~$\nu$ by removing~$r_a(\vec{m})$ parts of size~$a$ from~$\nu$ for all positive~$a$. Denote by~$n'=n-|\vec{m}|$ the size of~$\nu'$. By the Murnaghan--Nakayama rule one has
\[\chi^\lambda(\nu) \= \sum_{\xi\in \mathrm{BS}(\lambda,\vec{m})} \chi^{\xi}(\vec{m}) \, \chi^{\lambda\backslash\xi}(\nu').\]
One has
\[z_\nu^{-1}\prod_{a=1}^\infty \binom{r_{a}(\nu)}{r_{a}(\vec{m})}
\=\prod_{a=1}^\infty \frac{1}{a^{r_a(\nu)}r_a(\vec{m})!\,(r_a(\nu)-r_a(\vec{m})) !}
\=\frac{1}{z_{\nu'}z_{\vec{m}}}. \]
Hence,~(\ref{eq:mollerterm}) equals
\[\sum_{\xi\in \mathrm{BS}(\lambda,\vec{m})}\sum_{\rho\in \mathrm{BS}(\lambda,\vec{m})} \frac{\chi^{\xi}(\vec{m})\,\chi^{\rho}(\vec{m})}{z_{\vec{m}}} \sum_{\nu'\vdash n'} z_{\nu'}  \chi^{\lambda\backslash\xi}(\nu')\,\chi^{\lambda\backslash\rho}(\nu').\]
The orthogonality relation for the symmetric group is the statement 
\[\sum_{\nu'\vdash n'} z_{\nu'}  \,\chi^{\lambda\backslash\xi}(\nu')\,\chi^{\lambda\backslash\rho}(\nu')\=\delta_{\lambda\backslash\xi,\lambda\backslash\rho}\, .\]
Hence, we obtain the desired result. 
\end{proof}

The~$q$-bracket of an element in~$\mathcal{X}$ is not necesarily a quasimodular form. However, it always lies in the following space of~$q$-analogues of zeta values, see \cite{GKZ06}. 
\begin{defn}\label{defn:combeis}
Let~$\mathcal{C}_{\leq \ell}$ be the~$\q$-vector space consisting of all polynomials in the \emph{combinatorial Eisenstein series}
\[G_k(\tau) \= -\frac{B_k}{2k}+\sum_{r=1}^\infty\sum_{m=1}^\infty m^{k-1} q^{mr},  \quad\quad (k\geq 1, \text{ not necesarily even})\]
\emph{and their derivatives} of weight~$\leq \ell$, where to~$D^rG_k$ we assign the weight~$k+2r$. 
\end{defn}
Now, \cref{prop:Hkl} implies the following result:
\begin{thm}
For all~$f\in \mathcal{X}_{\leq k}$, one has~$\langle f \rangle_q \in \mathcal{C}_{\leq k}.$
\end{thm}
\begin{proof}
By \cref{prop:Hkl},~$f$ equals the M\"oller transform of some polynomial in the~$T_{k,l}$ with respect to the product~$\blok$. Here, however, it may happen that~$k+l$ is odd. Mutatis mutandis in either of three approaches in \S\S\ref{sec:3proofs}, we find that the~$q$-bracket of~$T_{k,l}$ lies in~$\mathcal{C}_{k+l}$, which proves the result. 
\end{proof}

\begin{thm} For all weights~$k$ one has~$\moller (\TA_k)\subset \mathcal{X}_{\leq k}$.  More precisely, 
\begin{align}\label{eq:molthm} \frac{\moller (T_{k_1,l_1}\blok\cdots \blok T_{k_n,l_n})}{(l_1-1)!\cdots(l_n-1)!} \= X_{\vec{k},\vec{l}} + \text{elements in } \mathcal{X} \text{ of lower degree}.\end{align}
\end{thm}
\begin{proof}
Observe that \cref{prop:Hkl} implies that $\moller (\TA_k)\subset \mathcal{X}_{\leq k}$. Equation~(\ref{eq:molthm}) follows from this proposition after noting that the M\"oller transformation of $T_{k,l} - (l-1)! U_{k,l}$ has degree strictly smaller than~$k+l$. 
\end{proof}
%The following elements of~$\mathcal{X}$ have a quasimodular~$q$-bracket.
\begin{exmp}
The following two equations provide examples of linear combinations of elements of~$\mathcal{X}$ with a quasimodular~$q$-bracket whenever~${k+l}$ and~$k_i$ are even integers. 
\begin{align}\moller(T_{k,l}) &\= -\frac{B_{k+l}}{2(k+l)}(\delta_{k,1}+\delta_{l,0})+\sum_{j=1}^l \stirling{l}{j}(j-1)! X_{k,j}\, , \\
\moller(S_{k_1}\blok S_{k_2}\blok\cdots\blok S_{k_n}) &\= \sum_{A\subset [n]}\Bigl(\,\prod_{i\not \in A} \frac{B_{k_i}}{2k_i}\Bigr) X_{k_A,(1,1,\ldots,1)}\, . \end{align}
See \cref{sec:a} for a table of elements in~$\mathcal{X}$ with quasimodular~$q$-bracket and  of small degree. 
%i.e., for~$n=2$ one has
%\begin{align}\moller(S_{k_1}\blok S_{k_2}) = & \frac{B_{k_1}B_{k_2}}{4k_1k_2} - \frac{B_{k_1}}{2k_1}\left(\sum_{m\in \n}\sum_{\xi \in \comm{}} \chi^{\xi}(m)^2 m^{k_2}\right)- \frac{B_{k_2}}{2k_2}\left(\sum_{m\in \n}\sum_{\xi \in \comm{}} \chi^{\xi}(m)^2 m^{k_1}\right)+\\
%&\sum_{m_1,m_2\in \n}\sum_{\xi \in \comm{}} \chi^{\xi}(m_1,m_2)^2 m_1^{k_1}m_2^{k_2}.\end{align}
\end{exmp}

\begin{remark} In many examples the~$X_{\vec{k},\vec{l}}$ are not shifted symmetric functions or generated by shifted symmetric functions under the induced product. For example,~$\moller(T_{0,2}) \neq \moller(S_2)$ and besides~$Q_2=\moller(S_2)=S_2$ there are no other non-trivial functions generated by~$\Lambda^*$ under the pointwise product. It remains an open question whether the elements of~$\mathcal{X}$ are in some sense related to shifted symmetric functions. %Moreover,~$\mathcal{X}$ is not a quasimodular algebra: products of two elements in~$\mathcal{X}$ do not have quasimodular~$q$-brackets in many examples.
\end{remark}

\paragraph{Moments of other partition invariants}\label{sec:morepolynomials} 
So far we provided many examples of functions on partitions in~$\Lambda^*$ and~$\TA$ related to the representation theory of the symmetric group. Now, we see that many purely combinatorial notions lead to different bases for~$\SA$. We compare these bases to corresponding bases of~$\Lambda^*$. Most of these bases take the following form. Suppose an index set~$I$ and a sequence~$\{s_i\}_{i\in I}^\infty$ of elements of~$\q^\partitions$ are given. Then, we define the~$k$th moment of~$\vec{s}$ by (whenever this sum converges)
\[M_k(\vec{s})(\lambda) \= \sum_{i\in I} \Bigl(s_i(\lambda)^k-s_i(\emptyset)^k\Bigr).\]
For example, let the functions~$\vec{p},\vec{q}$ for the index set~$\n$ be given by
\begin{equation}
p_i(\lambda)=\lambda_i\, ,\quad\quad q_i(\lambda)=\lambda_i-i.
\end{equation}
Then, by definition,
\[S_k \= S_k(\emptyset)+M_{k-1}(\vec{p}),\quad\quad Q_k \= Q_k(\emptyset)+M_{k-1}(\vec{q}).\]
Note that by definition~$M_k(\vec{s})(\emptyset)=0$. As the functions below will not respect the weight grading anyway, we will not include a constant term. 

The sequences~$\vec{a},\vec{c},\vec{h},\vec{x}$ of functions on partitions are of further interest. Define these sequence, indexed by $\xi=(i,j)\in \z_{\geq 0}^2$, by~$0$ if~$\xi \not \in Y_\lambda$ and 
\begin{align}
a_\xi(\lambda)&:\ \text{arm length of } \xi & h_\xi(\lambda)&:\ \text{hook-length of } \xi%, \text{i.e., } a_\xi(\lambda)+b_\xi(\lambda)-1
\\ 
%b_\xi(\lambda)&:\ \text{leg length of } \xi\\
x_\xi(\lambda)&\= i & c_\xi(\lambda)&:\ \text{content of } \xi,\text{ i.e., } i-j 
%y_\xi(\lambda)&=j
\end{align}
if~$\xi \in Y_\lambda$. For~$\vec{h}$ and~$\vec{c}$ it is known that the corresponding moment functions are shifted symmetric, for the latter see \cite[Theorem 4]{KO94}. The moment functions corresponding to~$\vec{a}$ and $\vec{x}$ turn out to be equal and to be elements of~$\SA$. 
\begin{thm}
\[\SA \= \q[M_k(\vec{a}) \mid k\geq 0 \text{ even}] \= \q[M_k(\vec{x}) \mid k\geq 	0 \text{ even}].\]
%The algebra generated by all~$M_k(\vec{s})$ where~$k$ runs over all even nonzero integers equals~$\SA$ if~$\vec{s}=\vec{a}$ or if~$\vec{s}=\vec{x}$. 
%The algebra~$\SA$ is generated by the~$M_k(\vec{s})$ for
%\[\begin{array}{l l l}
% \vec{s}(\lambda) & k \\\hline
%(\lambda_1,\lambda_2,\ldots) & \text{odd} \\
%(r_1(\lambda),r_2(\lambda),\ldots) & (1,2,\ldots) & \text{odd} \\
%(i \mid (i,j)\in Y_\lambda) &\text{even} \\
%(a(\xi) \mid \xi \in Y_\lambda) &\text{even} \\
%(\lambda_1',\lambda_2',\ldots) & (1,2,\ldots) & \text{even}
%\end{array}\]
\end{thm}

\begin{proof}
As the Faulhaber polynomials~$\Faulhaber_k$ with~$k$ odd form a basis for the space of all odd polynomials, the functions  
\[\sum_{i=1}^\infty \Faulhaber_k(\lambda_i) \= \sum_{i=1}^\infty \sum_{a=1}^{\lambda_i} a^{k-1}\]
generate~$\SA$, which corresponds to the first equality in the statement. By interchanging the sums one obtains
\begin{align}\label{eq:kerov2}
\sum_{i=1}^\infty \Faulhaber_k(\lambda_i) \=\sum_{a=1}^\infty  a^{k-1}\sum_{m=a}^\infty r_m(\lambda)  \= \sum_{(i,j)\in Y_\lambda} i^{k-1}.
\end{align}
Hence, the result is also true for~$\vec{s}=\vec{x}$. 
\end{proof}

\begin{remark}
Note that for a given~$i$ the number of~$(i,j)\in Y_\lambda$ equals~$\lambda_i'$, where~$\lambda'$ is the conjugate partition of~$\lambda$. Hence,~(\ref{eq:kerov2}) can be written as
\[\sum_{i=1}^\infty i^{k-1}\lambda_i'\]
and consequently these functions for~$k$ odd generate~$\SA$. Note that these functions are different from the~$S_k(\lambda')$. In fact, the algebra generated by the~$S_k(\lambda')$ is distinct from the algebra~$\SA$, in contrast to the algebra of shifted symmetric functions, for which~$Q_k(\lambda')=(-1)^kQ_k(\lambda)$. 
\end{remark}

\paragraph*{Acknowledgment}
I am very grateful for the conversations with and feedback received from both my supervisors Gunther Cornelissen and Don Zagier.

\appendix
\section{Table of double moment functions up to weight~\texorpdfstring{$4$}{4}}\label{sec:a}
For all basis elements~$f\in\TA_{\leq 4}$ in the basis provided by \cref{thm:basis}, we compute its representation in the basis consisting of double moment function and the quantities~$\langle f\rangle_{\vec{u}}, \langle f \rangle_q, D(f), \mathfrak{d}(f)$ and~$\moller(f)$.

\subsection*{Weight at most~\texorpdfstring{$2$}{2}} \renewcommand{\arraystretch}{1.3}
\definecolor{orange}{rgb}{1,0.975,0.95}
\definecolor{greenish}{rgb}{0.95,0.97,1}
\rowcolors{1}{orange}{greenish}
$\displaystyle \arraycolsep=9pt
\begin{array}[h!]{l || l | l | l}
f&1 & T_{1,1} &T_{0,2} \\
\langle f \rangle_{\vec{u}} & 1 &  -\frac{1}{24}+\sum_{m,r\geq 1} m u_m^r & -\frac{1}{24}+\sum_{m,r\geq 1} r u_m^r\\
\langle f \rangle_q & 1 &  G_2 & G_2 \\
D(f) & 0 & T_{2,2} & T_{1,3}\\
\mathfrak{d}(f) & 0 &  -\frac{1}{2} &  -\frac{1}{2}\\
\moller(f) & 1 &  X_{1,1}-\frac{1}{24} & X_{0,2}+X_{0,1}-\frac{1}{24}
\end{array}$

\subsection*{Weight \texorpdfstring{$4$}{4}}
$\displaystyle \arraycolsep=9pt
\begin{array}{l || l | l | l | l }
f&T_{3,1} &T_{2,2} &  T_{1,3} & T_{0,4} \\
\langle f \rangle_{\vec{u}} & \frac{1}{240}+\sum m^3 u_m^r & \sum m^2r u_m^r & \sum m r^2 u_m^r &  \frac{1}{240}+\sum r^3 u_m^r \\
\langle f \rangle_q & G_4 & \frac{5}{6}G_4-2G_2^2 & \frac{5}{6}G_4-2G_2^2 & G_4 \\
D(f) & T_{4,2} & T_{3,3} & T_{2,4} & T_{1,5}\\
\mathfrak{d}(f) & 0& 2T_{1,1} & 2T_{0,2} & 0\\
\moller(f) &X_{3,1} +\frac{1}{240} &  X_{2,2} + X_{2,1} & 2X_{1,3} +3X_{1,2}+X_{1,1} & 6 X_{0,4} + 12 X_{0,3} + \\    \rowcolor{greenish}
		& 						&	 					& 							& 7X_{0,2}+ X_{0,1}+\frac{1}{240}
\end{array}$\vspace{5pt}

\noindent
$\displaystyle \arraycolsep=9pt
\begin{array}{l || l | l }  \rowcolor{orange}
f &  T_{1,1}\blok  T_{1,1} = &  T_{1,1} \blok T_{0,2} = \\ 
 & T_{1,1}^2-T_{2,2} & T_{1,1}T_{0,2}-T_{1,3}  \\
\langle f \rangle_{\vec{u}} & \sum m_1m_2 u_{m_1}^{r_1}u_{m_2}^{r_2}-\frac{1}{12}\sum mu_m^r+\frac{1}{576} & \sum m_1r_2 u_{m_1}^{r_1}u_{m_2}^{r_2}-\frac{1}{24}\sum(m+r)u_m^r+\frac{1}{576}\\
\langle f \rangle_q  & G_2^2 & G_2^2 \\
D(f) & 2\cdot T_{2,2}\blok T_{1,1} & T_{2,2}\blok T_{0,2}+T_{1,1}\blok T_{1,3} \\
\mathfrak{d}(f) & -T_{1,1} & -\frac{1}{2}T_{1,1}-\frac{1}{2}T_{0,2} \\
\moller(f) & 
X_{(1,1),(1,1)}-\frac{1}{12}X_{1,1}+ \frac{1}{576} 
& X_{(1,0),(1,2)}+X_{(1,0),(1,1)} + \\  \rowcolor{greenish}
& & -\frac{1}{24}\left(X_{1,1}+X_{0,2}\right)+\frac{1}{576} 
\end{array}$\vspace{5pt}

\noindent
$\displaystyle \arraycolsep=9pt
\begin{array}{l || l }  \rowcolor{orange}
f &  T_{0,2} \blok T_{0,2} = \\ 
& T_{0,2}^2-\frac{5}{6}T_{0,4}-\frac{1}{6}T_{0,2}-\frac{1}{288}	 \\
\langle f \rangle_{\vec{u}} & \sum r_1r_2 u_{m_1}^{r_1}u_{m_2}^{r_2}-\frac{1}{12}\sum ru_m^r+\frac{1}{576}\\
\langle f \rangle_q &  G_2^2 \\
D(f) & 2\cdot T_{1,3}\blok T_{0,2}\\
\mathfrak{d}(f)  & -T_{0,2} \\
\moller(f) & X_{(0,0),(2,2)}+X_{(0,0),(2,1)}+X_{(0,0),(1,2)}+\\  \rowcolor{greenish}
& X_{(0,0),(1,1)}-\frac{1}{12}X_{0,2}-\frac{1}{12}X_{0,1}+ \frac{1}{576} 
\end{array}$

%\[
%\begin{array}{l l l}
%f & \moller f & \langle f\rangle_q \\ \hline 	
% T_{1,1} & X_{1,1}-\frac{1}{24} & G_2 \\
% T_{0,2} & X_{0,2}+X_{0,1}-\frac{1}{24}  &G_2 \\
% T_{3,1} & X_{3,1} +\frac{1}{240}& G_4 \\
% T_{2,2} & X_{2,2} + X_{2,1} & \frac{5}{6}G_4-2G_2^2 \\
%  T_{1,3} & 2X_{1,3} +3X_{1,2}+X_{1,1} &\frac{5}{6}G_4-2G_2^2 \\
%  T_{0,4} & 6 X_{0,4} + 12 X_{0,3} + 7X_{0,2}+X_{0,1}+\frac{1}{240}& G_4  \\
% T_{1,1}\blok  T_{1,1} = T_{1,1}^2+T_{2,2} & X_{(1,1),(1,1)}-\frac{1}{12}X_{1,1}+ \frac{1}{576}& G_2^2  \\
% T_{1,1} \blok T_{0,2} = T_{1,1}T_{0,2}+T_{1,3}+T_{1,1} & X_{(1,0),(1,2)}+X_{(1,0),(1,1)}-\frac{1}{24}\left(X_{1,1}+X_{0,2}\right)+\frac{1}{576}  & G_2^2 \\
% T_{0,2} \blok T_{0,2} = T_{0,2}^2+T_{0,4}+T_{0,2} & X_{(0,0),(2,2)}+X_{(0,0),(2,1)}+X_{(0,0),(1,2)}+X_{(0,0),(1,1)}-\frac{1}{12}X_{0,2}+ \frac{1}{576}  & G_2^2  
%\end{array}\]

\small
%\bibliographystyle{alpha}
%\bibliography{../../bib}{}

\end{document}